\documentclass[letterpaper,12pt]{article}
\usepackage{amsmath, amsthm, amsfonts,amssymb}
\usepackage{fourier}
\usepackage[usenames,dvipsnames]{color}
\usepackage[all,cmtip]{xy}
\usepackage[verbose,colorlinks=true,linktocpage=true,linkcolor=blue,citecolor=blue]{hyperref}

\theoremstyle{definition}
\newtheorem{definition}{Definition}[section]
\theoremstyle{plain}
\newtheorem{theorem}[definition]{Theorem}
\newtheorem{proposition}[definition]{Proposition}
\newtheorem{lemma}[definition]{Lemma}

\theoremstyle{remark}
\newtheorem{remark}[definition]{Remark}

\numberwithin{equation}{section}

\newcommand{\Uq}{\mathrm{U}_q}

\allowdisplaybreaks[1]

\begin{document}

\title{The first fundamental theorem of invariant theory for the quantum queer superalgebra}
\author{Zhihua Chang ${}^1$ and Yongjie Wang \footnote{Corresponding Author: Y. Wang (Email: wyjie@mail.ustc.edu.cn)}${}^2$}
\maketitle

\begin{center}
\footnotesize
\begin{itemize}
\item[1] School of Mathematics, South China University of Technology, Guangzhou, Guangdong, 510640, China.
\item[2] School of Mathematics, Hefei University of Technology, Hefei, Anhui, 230009, China.
\end{itemize}
\end{center}
\begin{abstract}
The classical invariant theory for the queer Lie superalgebra is an investigation of the $\mathrm{U}(\mathfrak{q}_n)$-invariant sub-superalgebra of the symmetric superalgebra $\mathrm{Sym}(V^{\oplus r}\oplus V^{*\oplus s})$ for $V=\mathbb{C}^{n|n}$. We establish the first fundamental theorem of invariant theory for the quantum queer superalgebra $\Uq(\mathfrak{q}_n)$. The key ingredient is a quantum analog $\mathcal{O}_{r,s}$ of the symmetric superalgebra $\mathrm{Sym}(V^{\oplus r}\oplus V^{*\oplus s})$ that is created as a braided tensor product of a quantization $\mathsf{A}_{r,n}$ of $\mathrm{Sym}(V^{\oplus r})$ and a quantization $\bar{\mathsf{A}}_{s,n}$ of $\mathrm{Sym}(V^{*\oplus s})$. Since the quantum queer superalgebra $\Uq(\mathfrak{q}_n)$ is not quasi-triangular, our braided tensor product is created via an explicit intertwining operator instead of the universal $\mathcal{R}$-matrix. 
\bigskip

\noindent\textit{MSC(2020):} 17B37, 16T20, 20G42.
\bigskip

\noindent\textit{Keywords:} Quantum queer superalgebra; Howe duality; Invariant theory.
\end{abstract}

\section{Introduction}
Let $G$ be a classical group and $M$ a commutative algebra on which $G$ acts as automorphisms. Then the $G$-invariant subspace $$M^G=\big\{v\in M| gv=v, \ \ \forall g\in G\big\}$$
is a subalgebra of $M$. The first and second fundamental theorems of classical invariant theory describe generators and relations for the algebra $M^G$,  we refer reader to \cite{Howe95} for details. The fundamental theorems of the invariant theory of a group is an indispensable part of its representation theory. They also play significant roles in other branches of mathematics and physics, including the construction of the Jones polynomial of knots in the topological quantum field theory.

The fundamental theorems of invariant theory are usually formulated in three equivalent ways: the linear
formulation on tensor modules\cite{Howe95},  the commutative algebraic formulation on symmetric algebras of multi-copies of the natural modules in  \cite{CW12, LZZ11} and the
Schur-Weyl duality formulation \cite{Weyl}. These formulations of the
fundamental theorems can be described in a unified way  by using the category
of Brauer diagrams \cite{DLZ18, LZZ20, LZ12, LZ15, LZ17, LZ21, XYZ20}. 

We are particularly interested in the commutative algebraic formulation which deals with the invariants in the situation where $M=\mathrm{Sym}(V^{\oplus r}\oplus V^{*\oplus s})$ for the natural $G$-module $V$ and its dual $V^*$. A systematic method to demonstrate the fundamental theorems in this situation had been developed in \cite{Howe95}. As a super analog of classical invariant theory,  Sergeev \cite{Sergeev01} explicitly described the invariants in the $\mathrm{U}(\mathfrak{g})$-supermodule $\mathrm{Sym}(V^{\oplus r}\oplus V^{*\oplus s})$ for a finite-dimensional simple Lie superalgebra $\mathfrak{g}\subseteq\mathfrak{gl}(V)$.

Mathematicians have also made remarkable progress on the fundamental theorems of  invariant theory for quantum groups in recent decades. The first fundamental theorem for the quantum linear group $\Uq(\mathfrak{gl}_n)$ was established in \cite{LZZ11}, in which a quantization of $\mathrm{Sym}(V^{\oplus r}\oplus V^{*\oplus s})$ was given as a braided tensor product of a quantum coordinate algebra and a dual quantum coordinate algebra. The universal $\mathcal{R}$-matrix of $\Uq(\mathfrak{gl}_n)$ played a crucial role in defining the braided tensor product. For the quantum general linear superalgebra $\Uq(\mathfrak{gl}_{m|n})$, this problem was solved in \cite{Zhangyang20}. The authors in \cite{LZZ20} investigated the first fundamental theorem of invariant theory for quantum ortho-symplectic supergroups from a new apporach by using the technique of ribbon tensor categories and Etingof-Kazhdan quantization \cite{Ge06}.

The purpose of this paper is to establish the first fundamental theorem (FFT) of  invariant theory for the quantum queer superalgebra $\Uq(\mathfrak{q}_n)$.
 The quantum queer superalgebra  was constructed in Olshanski's letter \cite{Olshanski92} with the Faddeev-Reshetikhin-Takhtajan (FRT) formalism. based on which a celebrated Schur-Weyl duality, often referred to as Sergeev-Olshanski duality, had been established.
 It has been further generalized to mixed tensor powers in \cite{BGJKW16}.The highest weight representation theory and their crystal basis theory  for quantum queer superalgebra  was investigated in \cite{GJKK10, GJKK14, GJKKK15}. The authors established the Howe duality for quantum queer superalgebra by introducing quantum coordinate superalgebra \cite{CW20}, which is isomorphic to a braided symmetric superalgebra \cite{BDK20}.  Also in \cite{BDK20}, they define a monoidal supercategory of quantum type $Q$ webs and show it admits a full, essentially surjective
functor onto the monoidal supercategory of $\mathrm{U}_q(\mathfrak{q}_n)$-supermodules generated by the quantum symmetric powers of the natural module and their duals. Recently, Savage in \cite{Savage22} provided categorical tools
for studying the representation theory of the quantum queer (called isomeric) superalgebras by introducing the quantum queer supercategory and the quantum affine queer supercategory.

The queer Lie superalgebras form a family of ``strange'' Lie superalgebras since they do not possess non-degenerate invariant even bilinear forms.  The strange phenomena is also observed in quantum setting. The quantum queer superalgebra $\mathrm{U}_q(\mathfrak{q}_n)$ is not a quasi-triangular Hopf superalgebra because of the absence of a universal $\mathcal{R}$-matrix. Thus the existence of an intertwining operator between two finite-dimensional $\mathrm{U}_q(\mathfrak{q}_n)$-supermodules are not guaranteed except the polynomials modules. Recently, Savege has affirmed the existence of an intertwining operator between a finite-dimensional $\Uq(\mathfrak{q}_n)$-supermodule and the natural module $V$. Another remarkable difference between the quantum qeer superalgebra and the quantum general linear superalgebra is that the tensor product of two nature modules is not irreducible as an $\mathrm{U}_q(\mathfrak{q}_n)\otimes\mathrm{U}_q(\mathfrak{q}_n)$-supermodule. Similar ``strange'' phenomena is also observed \cite{AGG21} in periplectic Lie superalgebras and their quantum enveloping superalgebras.
 
 Our approach to obtain the first fundamental theorem for quantum queer superalgebras takes advantage of the Howe duality between the quantum queer superalgebras $\Uq(\mathfrak{q}_r)$ and $\Uq(\mathfrak{q}_n)$. According to the Howe duality, there is an associative superalgebra $\mathsf{A}_{r,n}$ (a quantization of $\mathrm{Sym}(V^{\oplus r})$) that admits a multiplicity free decomposition as a $\Uq(\mathfrak{q}_r)\otimes\Uq(\mathfrak{q}_n)$-supermodules. A quantization of $\mathrm{Sym}(V^{\oplus *s})$ is also obtained in the same manner. These lead to a $\Uq(\mathfrak{q}_n)$-supermodule $\mathcal{O}_{r,s}=\mathsf{A}_{r,n}\otimes\bar{\mathsf{A}}_{s,n}$.

In order to define a braided multiplication on $\mathcal{O}_{r,s}$ such that it is compatible with the $\Uq(\mathfrak{q}_n)$-action, we need an intertwining operator
$$\Upsilon:\bar{\mathsf{A}}_{s,n}\otimes\mathsf{A}_{r,n}\rightarrow\mathsf{A}_{r,n}\otimes\bar{\mathsf{A}}_{s,n}.$$
For those quantum superalgebra $\Uq(\mathfrak{g})$ that is a quasi-triangular Hopf superalgebra (see \cite{Yamane91,Yamane94}), such a intertwining operator is the composition of the permutation operator $P$ and the canonical image of the universal $\mathcal{R}$-matrix. The universal $\mathcal{R}$-matrix plays important roles in the study of the fundamental theorems of invariant theory for a quasi-triangular quantum group (supergroups) as in \cite{LZZ11, Zhangyang20}, as well as the proof of the Harish-Chandra theorem for these quantum superalgebras $\mathrm{U}_q(\mathfrak{g})$ of type $A$-$G$ in \cite{LWY22}. 

The main obstacle for the quantum queer superalgebra $\Uq(\mathfrak{q}_n)$ is the absence of a universial $\mathcal{R}$-matrix. Instead, we will explicit create an intertwining operator $\Upsilon$ without using the universal $\mathcal{R}$-matrix. The definition of $\Uq(\mathfrak{q}_n)$ involves an intertwining operator $\check{S}=P\circ S:V\otimes V\rightarrow V\otimes V$ on the tensor product of the natural $\Uq(\mathfrak{q}_n)$-module, we will show that the intertwining operator $\check{S}$ indeed determines an intertwining operator $\Upsilon:\bar{\mathsf{A}}_{s,n}\otimes\mathsf{A}_{r,n}\rightarrow\mathsf{A}_{r,n}\otimes\bar{\mathsf{A}}_{s,n}$. We point out that calculation with generating matrices instead of generators effectively simplifies the tedious verification on the well-definedness of $\Upsilon$ and subsequent discussion. Most calculations in this paper will be conducted using generating matrices. It seems to be a neat tool for dealing with the quantum algebras with Faddeev-Reshetikhin-Takhtajan presentation.

Once the intertwining operator $\Upsilon:\bar{\mathsf{A}}_{s,n}\otimes\mathsf{A}_{r,n}\rightarrow\mathsf{A}_{r,n}\otimes\bar{\mathsf{A}}_{s,n}$ is given, we obtain a $\Uq(\mathfrak{q}_n)$-supermodule superalgebra 
  $\mathcal{O}_{r,s}=\mathsf{A}_{r,n}\otimes\bar{\mathsf{A}}_{s,n}$ that is a quantization of $\mathrm{Sym}\left(V^{\oplus r}\oplus V^{*\oplus s}\right)$. It is also equipped with the $\Uq(\mathfrak{q}_r)\otimes\Uq(\mathfrak{q}_s)$-action commuting with the $\Uq(\mathfrak{q}_n)$-action. If $n\geqslant\max(r,s)$, we show that the $\Uq(\mathfrak{q}_n)$-invariant sub-superalgebra $\left(\mathcal{O}_{r,s}\right)^{\Uq(\mathfrak{q}_n)}$ is a $\Uq(\mathfrak{q}_r)\otimes\Uq(\mathfrak{q}_s)$-supermodule that has the same multiplicity free decomposition as $\mathsf{A}_{r,s}$. This motivates us to define a $\mathbb{C}(q)$-linear map $\tilde{\Delta}_{r,s}:\mathsf{A}_{r,s}\rightarrow \left(\mathcal{O}_{r,s}\right)^{\Uq(\mathfrak{q}_n)}$. Then the first fundamental theorem follows from a careful analysis on the behaviors of $\tilde{\Delta}_{r,s}$ when $\Uq(\mathfrak{q}_r)$ and $\Uq(\mathfrak{q}_s)$ act on $\mathsf{A}_{r,s}$, as well as we multiply elements in $\mathsf{A}_{r,s}$. If $n<\max(r,s)$, the problem is reduced to the $\Uq(\mathfrak{q}_n)$-invariant sub-superalgebra $\left(\mathcal{O}_{m,m}\right)^{\Uq(\mathfrak{q}_n)}$ with $m=\min(r,s,n)$ since $\left(\mathcal{O}_{r,s}\right)^{\Uq(\mathfrak{q}_n)}$ is generated by $\left(\mathcal{O}_{m,m}\right)^{\Uq(\mathfrak{q}_n)}$ as a $\Uq(\mathfrak{q}_r)\otimes\Uq(\mathfrak{q}_s)$-supermodule. The second fundamental theorem requires a meticulous discussion on the $\Uq(\mathfrak{q}_r)\otimes\Uq(\mathfrak{q}_s)$-supermodule structure of $\left(\mathcal{O}_{r,s}\right)^{\Uq(\mathfrak{q}_n)}$. We will treat it in a upcoming paper \cite{CW22}.

 This paper is organized as follows. We review some basic facts related to the Howe duality of quantum queer superalgebras in Section~\ref{queer}. Then we explicitly construct the intertwining operator $\Upsilon$ in Section~\ref{braid}, that leads to a $\Uq(\mathfrak{q}_n)$-module superalgebraic structure on $\mathcal{O}_{r,s}=\mathsf{A}_{r,n}\otimes\bar{\mathsf{A}}_{s,n}$. Section~\ref{Invariant} serves for finding $\Uq(\mathfrak{q}_n)$-invariant elements in $\mathcal{O}_{r,s}$ and calculating relations satisfied by these invariant elements. Then we establish the first fundamental theorem of the invariant theory for quantum queer superalgebra in Section~\ref{FFT}.

\section{Quantum queer superalgebras and their coordinates superalgebras}
\label{queer}
We always assume that the base field $\mathbb{C}(q)$ is the field of rational functions in an indeterminate $q$. For a positive integer $N$, we denote $I_{N|N}:=\big\{-N,\ldots,-1,1,\ldots,N\big\}$. Let $V$ be the $2N$-dimensional $\mathbb{C}(q)$-vector space with the basis $\{v_i, i\in I_{N|N}\}$, which is equipped with a $\mathbb{Z}/2\mathbb{Z}$-grading
$$|v_i|=|i|:=\begin{cases}\bar{0},&\text{if }i>0,\\ \bar{1}&\text{if }i<0.\end{cases}$$
Then $\mathrm{End}(V)$ is naturally an associative superalgebra, in which the standard matrix unit $E_{ij}$ is of the parity $|i|+|j|$ for $i,j\in I_{N|N}$. 

As in \cite{Olshanski92}, we set
\begin{align}
S:=&\sum\limits_{i,j\in I_{N|N}}q^{\varphi(i,j)}E_{ii}\otimes E_{jj}+\xi \sum\limits_{i<j }(-1)^{|i|}(E_{ji}+E_{-j,-i})\otimes E_{ij}\in\mathrm{End}(V)^{\otimes 2}\label{eq:smatrix}\\
=&\sum\limits_{i,j\in I_{N|N}}S_{ij}\otimes E_{ij},\nonumber
\end{align}
where $\varphi(i,j)=(-1)^{|j|}(\delta_{i,j}+\delta_{i,-j}) \text{ and }\xi=q-q^{-1}$. It satisfies the quantum Yang-Baxter equation:
$$S^{12}S^{13}S^{23}=S^{23}S^{13}S^{12},$$
where 
$$S^{12}=S\otimes 1,\quad S^{23}=1\otimes S,\quad S^{13}=\sum\limits_{i,j\in I_{N|N}}S_{ij}\otimes 1\otimes E_{ij}.$$
\textit{The quantum queer superalgebra} is defined via the Faddeev-Reshetikhin-Takhtajan presentation as follows:

\begin{definition}[G. Olshanski \cite{Olshanski92}]
The quantum queer superalgebra $\Uq(\mathfrak{q}_N)$ is the unital associative superalgebra over $\mathbb{C}(q)$ generated by elements $L_{ij}$ of parity $|i|+|j|$ for $i,j\in I_{N|N}$ and $i\leqslant j$, with defining relations:
\begin{eqnarray}
&L_{ii}L_{-i,-i}=1=L_{-i,-i}L_{ii},&\\
&L^{[1]2}L^{[1]3}S^{23}=S^{23}L^{[1]3}L^{[1]2},&\label{eq:SLL}
\end{eqnarray}
where $L^{[1]2}=\sum\limits_{i\leqslant j} L_{ij}\otimes E_{ij}\otimes 1$, $L^{[1]3}=\sum\limits_{i\leqslant j} L_{ij}\otimes1\otimes E_{ij}$ and the relation \eqref{eq:SLL} holds in $\Uq(\mathfrak{q}_N)\otimes\mathrm{End}(V)\otimes\mathrm{End}(V)$. 
\end{definition}

The associative superalgebra $\Uq(\mathfrak{q}_N)$ is a Hopf superalgebra with the comultiplication $\Delta$, the counit $\varepsilon$  and the antipode $\mathcal{S}$ given respectively by:
\begin{equation*}
\Delta(L)=L\otimes L,\qquad \varepsilon(L)=1,\qquad \mathcal{S}(L)=L^{-1}.
\end{equation*}

Let $\Uq(\mathfrak{q}_N)^{\circ}$ denote the finite dual of the Hopf superalgebra $\Uq(\mathfrak{q}_N)$, i.e., 
\begin{equation*}
\Uq(\mathfrak{q}_N)^{\circ}:=\left\{f\in\Uq(\mathfrak{q}_N)^*\middle|\ \ker f\text{ contains a co-finite } \mathbb{Z}_2\text{-graded ideal of }\Uq(\mathfrak{q}_N)\right\}.
\end{equation*}
It is also a Hopf superalgebra with the multiplication $m^{\circ}$, the comultiplication $\Delta^{\circ}$, the counit $\varepsilon^{\circ}$ and the antipode $\mathcal{S}^{\circ}$ dualizing the Hopf superalgebra structure of $U_q(\mathfrak{q}_N)$. Namely, the canonical pairing $\langle\cdot\,,\cdot\rangle: \Uq(\mathfrak{q}_N)^{\circ}\times\Uq(\mathfrak{q}_N)\rightarrow\mathbb{C}(q)$ satisfies
\begin{align*}
\langle f, uu'\rangle=&\sum\limits_{(f)}(-1)^{|f_{(2)}||u|}\langle f_{(1)}, u\rangle\langle f_{(2)}, u'\rangle,
&&f\in\Uq(\mathfrak{q}_N)^{\circ}, u,u'\in\Uq(\mathfrak{q}_N),\\
\langle ff',u\rangle=&\sum\limits_{(u)}(-1)^{|f'||u_{(1)}|}\langle f,u_{(1)}\rangle\langle f', u_{(2)}\rangle,
&&f,f'\in\Uq(\mathfrak{q}_N)^{\circ}, u\in\Uq(\mathfrak{q}_N).
\end{align*}

There are two $\Uq(\mathfrak{q}_N)$-supermodule structures on $\Uq(\mathfrak{q}_N)^{\circ}$ given respectively by:
\begin{align*}
\langle \Phi_u(f),u'\rangle=&(-1)^{|u|}\langle f, u'u\rangle,\\
\langle \Psi_u(f),u'\rangle=&(-1)^{|f||u|}\langle f, \mathcal{S}(u)u'\rangle.
\end{align*}
We say that $\Uq(\mathfrak{q}_N)^{\circ}$ is a $\Uq(\mathfrak{q}_N)$-\textit{supermodule superalgebra} under the action $\Phi$ since the action $\Phi$ is compatible with the multiplication on $\Uq(\mathfrak{q}_N)^{\circ}$ in the sense that
\begin{equation*}\Phi_u(fg)=\sum(-1)^{|f||u_{(2)}|}\Phi_{u_{(1)}}(f)\Phi_{u_{(2)}}(g),\quad u\in\Uq(\mathfrak{q}_N), f,g\in\Uq(\mathfrak{q}_N)^{\circ}.\end{equation*}
The action $\Psi$ is compatible with the multiplication of $\Uq(\mathfrak{q}_N)^{\circ}$ in a different way as follows:
\begin{equation*}\Psi_u(fg)=\sum(-1)^{|f||u_{(1)}|+|u_{(1)}||u_{(2)}|}\Psi_{u_{(2)}}(f)\Psi_{u_{(1)}}(g),\quad u\in\Uq(\mathfrak{q}_N), f,g\in\Uq(\mathfrak{q}_N)^{\circ}.
\end{equation*}
We say that $\Uq(\mathfrak{q}_N)^{\circ}$ is a $\Uq(\mathfrak{q}_N)^{\mathrm{cop}}$-supermodule superalgebra under the action $\Psi$, where $\Uq(\mathfrak{q}_N)^{\mathrm{cop}}$ is the opposite co-superalgebra of $\Uq(\mathfrak{q}_N)$ that shares the same associative superalgebra structure with $\Uq(\mathfrak{q}_N)$ and is equipped with the opposite comultiplication. Moreover, $\Uq(\mathfrak{q}_N)^{\circ}$ is a $\Uq(\mathfrak{q}_N)\otimes\Uq(\mathfrak{q}_N)$-supermodule under the action $\Psi\otimes\Phi$ since
$$\Phi_u\Psi_{u'}=(-1)^{|u||u'|}\Psi_{u'}\Phi_u,\quad u,u'\in\Uq(\mathfrak{q}_N).$$

The $\mathbb{C}(q)$-vector space $V$ is naturally a $\Uq(\mathfrak{q}_N)$-supermodule via the homomorphism 
$$\Uq(\mathfrak{q}_N)\rightarrow\mathrm{End}(V), \quad L\mapsto S,$$
where $S$ is the matrix given in \eqref{eq:smatrix}. It gives rise to a family of matrix elements $t_{ia}\in\Uq(\mathfrak{q}_N)^{\circ}$ for $i,a\in I_{N|N}$  determined by
$$u.v_a=\sum_{i\in I_{N|N}}\langle t_{ia}, u\rangle v_i,\quad u\in\Uq(\mathfrak{q}_N).$$
The \textit{quantum coordinate superalgebra} $\mathsf{A}(\mathfrak{q}_N)$ is the sub-superalgebra of $\Uq(\mathfrak{q}_N)^{\circ}$ generated by $1$ and $t_{ia}, i,a\in I_{N|N}$. We remind the reader that $\mathsf{A}(\mathfrak{q}_N)$ is a sub-bi-superalgebra of $\Uq(\mathfrak{q}_N)^{\circ}$ (see \cite{CW20}), but not a Hopf sub-superalgebra since it is not closed under the antipode $\mathcal{S}^{\circ}$.

The superalgebra $\mathsf{A}(\mathfrak{q}_N)$ is invariant under both $\Uq(\mathfrak{q}_N)$-actions $\Phi$ and $\Psi$. For $1\leqslant r,n\leqslant N$, the sub-superalgebra of $\Uq(\mathfrak{q}_N)$ generated by $L_{ij}$ with $i,j\in I_{r|r}$ (resp. $i,j\in I_{n|n}$) is a Hopf sub-superalgebra of $\Uq(\mathfrak{q}_N)$ isomorphic to $\Uq(\mathfrak{q}_r)$ (resp. $\Uq(\mathfrak{q}_n)$). Then $\mathsf{A}(\mathfrak{q}_N)$ is a $\Uq(\mathfrak{q}_r)\otimes\Uq(\mathfrak{q}_n)$-supermodule by restricting the action $\Psi$ to $\Uq(\mathfrak{q}_r)$ and  the action $\Phi$ to $\Uq(\mathfrak{q}_n)$, and the sub-superalgebra $\mathsf{A}_{r,n}$ of $\mathsf{A}(\mathfrak{q}_N)$ generated by $t_{ia}$ with $i\in I_{r|r}$ and $a\in I_{n|n}$ is invariant under this restricted $\Uq(\mathfrak{q}_r)\otimes\Uq(\mathfrak{q}_n)$-action. Following \cite[Proposition~3.3]{CW20} and \cite[Proposition~3.6.4]{BDK20}, the superalgebra $\mathsf{A}_{r,n}$ is presented by the generators $t_{ia}, i\in I_{r|r}, a\in I_{n|n}$ and the relations
\begin{align}
t_{ia}&=t_{-i,-a},\label{eq:QCA1}\\
S^{12}T^{1[3]}T^{2[3]}&=T^{2[3]}T^{1[3]}S^{12},\label{eq:QCA2}
\end{align}
where $T=\sum\limits_{i\in I_{r|r}, a\in I_{n|n}}E_{ia}\otimes t_{ia}$.  It is a quantization of the symmetric superalgebra $\mathrm{Sym}\left((\mathbb{C}^{n|n})^{\oplus r}\right)$.

\begin{remark}
The matrix $T$ is not necessarily a square matrix, but the matrix multiplication works well in the following sense: We denote the identity matrices of size $2r\times 2r$ and size $2n\times 2n$ by $1_{r|r}$ and $1_{n|n}$, respectively. Then $T^{1[3]}=\sum E_{ia}\otimes 1_{r|r}\otimes t_{ia}$ and $T^{2[3]}=\sum 1_{n|n}\otimes E_{ia}\otimes t_{ia}$ on the left hand side. While $T^{1[3]}=\sum E_{ia}\otimes 1_{n|n}\otimes t_{ia}$ and $T^{2[3]}=\sum 1_{r|r}\otimes E_{ia}\otimes t_{ia}$ on the right hand side.
\end{remark}

Since the $\Uq(\mathfrak{q}_n)$-action $\Phi$ is compatible with the multiplication in $\mathsf{A}_{r,n}$, the action $\Phi$ is determined once the action of the generators $L_{ab}$ of $\Uq(\mathfrak{q}_n)$ on the generators $t_{ic}$ of $\mathsf{A}_{r,n}$ are given. These actions can be written in the following matrix form:
\begin{equation}
\begin{aligned}
L^{[2]3}\underset{\Phi}{\cdot}T^{1[2]}=&
\sum\limits_{a,b\in I_{n|n}}\sum\limits_{i\in I_{r|r}, c\in I_{n|n}}(-1)^{(|a|+|b|)(|i|+|a|)}E_{ia}\otimes\Phi_{L_{ab}}(t_{ic})\otimes E_{ab}\\
=&
T^{1[2]}S^{13},
\end{aligned}
\end{equation}
where $S$ is the matrix \eqref{eq:smatrix} for $\Uq(\mathfrak{q}_n)$. For the $\Uq(\mathfrak{q}_r)$-action $\Psi$ on $\mathsf{A}_{r,n}$, we similarly have
\begin{equation}
L^{[2]3}\underset{\Psi}{\cdot}T^{1[2]}=\left(S^{-1}\right)^{13}T^{1[2]},
\end{equation}
where 
\begin{equation}
S^{-1}=\sum\limits_{i,j\in I_{r|r}}q^{-\varphi(i,j)}E_{ii}\otimes E_{jj}-\xi \sum\limits_{i,j\in I_{r|r}, i<j }(-1)^{|i|}(E_{ji}+E_{-j,-i})\otimes E_{ij}\label{eq:sinverse}
\end{equation}
is the inverse of the matrix $S$ in \eqref{eq:smatrix}.

Now, the superalgebra $\mathsf{A}_{r,n}$ is a $\Uq(\mathfrak{q}_r)\otimes\Uq(\mathfrak{q}_n)$-supermodule under the action $\Psi\otimes\Phi$. It has been shown in \cite{CW20} that it admits a multiplicity free decomposition:
\begin{align}\label{eq:Adecom}
\mathsf{A}_{r,n}=\bigoplus_{\lambda\in\mathrm{SP}\atop\ell(\lambda)\leqslant\min(r,n)}\mathsf{L}_r^*(\lambda)\circledast \mathsf{L}_n(\lambda),
\end{align}
where $\mathrm{SP}=\{\lambda_1>\lambda_2>\cdots>\lambda_{\ell(\lambda)}>0\}$ is the set of strict partitions, $\ell(\lambda)$ is the length of $\lambda$, $\mathsf{L}_r(\lambda)$ is the highest weight $\Uq(\mathfrak{q}_r)$-supermodule of highest weight $\lambda$,  $\mathsf{L}_r^*(\lambda)$ is the dual of $\mathsf{L}_r(\lambda)$, and $\mathsf{L}_r^*(\lambda)\circledast\mathsf{L}_n(\lambda)$ is the unique irreducible $\Uq(\mathfrak{q}_r)\otimes\Uq(\mathfrak{q}_n)$-sub-supermodule\footnote{As a $\Uq(\mathfrak{q}_r)\otimes\Uq(\mathfrak{q}_n)$-supermodule, $\mathsf{L}_r^*(\lambda)\otimes\mathsf{L}_n(\lambda)$ is irreducible if $\ell(\lambda)$ is even, but it is a direct sum of two isomorphic irreducible sub-supermodules if $\ell(\lambda)$ is odd.} of $\mathsf{L}_r^*(\lambda)\otimes\mathsf{L}_n(\lambda)$. 
\bigskip

Note that the quantum coordinate superalgebra $\mathsf{A}_N$ is not closed under the antipode $\mathcal{S}^{\circ}$. We consider the dual $\Uq(\mathfrak{q}_N)$-supermodule $V^*$, whose matrix elements $\bar{t}_{\alpha b}, \alpha,b\in I_{N|N}$ together with $1$ generate another sub-superalgebra $\bar{\mathsf{A}}_N$ of $\Uq(\mathfrak{q}_N)^{\circ}$, called \textit{the dual quantum coordinate superalgebra}. The matrix element $\bar{t}_{\alpha b}$ is indeed determined by
$$\left\langle \bar{t}_{\alpha b}, u\right\rangle=(-1)^{|b|(|\alpha|+|b|)}\left\langle t_{b\alpha}, \mathcal{S}(u)\right\rangle,\quad \alpha,b\in I_{N|N}, u\in \Uq(\mathfrak{q}_N),$$
where $\mathcal{S}$ is the antipode on $\Uq(\mathfrak{q}_N)$. The superalgebra $\bar{\mathsf{A}}(\mathfrak{q}_N)$ is connected to $\mathsf{A}(\mathfrak{q}_N)$ by the antipode $\mathcal{S}^{\circ}$ of $\Uq(\mathfrak{q}_N)^{\circ}$ in the following way: 
\begin{equation*}
\mathcal{S}^{\circ}(t_{ia})=(-1)^{|i|(|i|+|a|)}\bar{t}_{ai},\text{ and }
\mathcal{S}^{\circ}(\bar{t}_{ia})=(-1)^{|i||a|+|a|}q^{(-1)^{|a|}2a-(-1)^{|i|}2i}t_{ai},
\end{equation*}
for $i,a\in I_{N|N}$.
\begin{remark}
The sub-superalgebra $\mathbf{Fun}(\Uq(\mathfrak{q}_N))$ of $\Uq(\mathfrak{q}_N)^{\circ}$ generated by $\mathsf{A}(\mathfrak{q}_N)$ and $\bar{\mathsf{A}}(\mathfrak{q}_N)$ is a sub-Hopf-superalgebra of $\Uq(\mathfrak{q}_n)^{\circ}$, which can be thought of as the superalgebra of regular functions of the quantum queer supergroup. Besides the relations \eqref{eq:QCA1}, \eqref{eq:QCA2}, \eqref{eq:DQCA1} and \eqref{eq:DQCA2}, the generators $t_{ia}$'s and $\bar{t}_{ia}$'s also satisfy
\begin{equation*}
\sum\limits_{p\in I_{N|N}}(-1)^{|i|(|j|+|p|)}t_{ip}\bar{t}_{jp}
=\sum\limits_{p\in I_{N|N}}(-1)^{(|i|+|p|)(|j|+1)}\bar{t}_{pi}t_{pj}
=\delta_{ij},
\end{equation*}
for $i,j\in I_{N|N}$. However, a complete list of defining relations for $\mathbf{Fun}(\Uq(\mathfrak{q}_N))$ is unknown yet.
\end{remark}

In order to create a quantization of $\mathrm{Sym}\left((\mathbb{C}^{n|n})^{*\oplus s}\right)$, we choose $N\geqslant s,n$ and consider the sub-superalgebra $\bar{\mathsf{A}}_{s,n}$ of $\bar{\mathsf{A}}(\mathfrak{q}_N)$ generated by $1$ and $\bar{t}_{\alpha b}$ with $\alpha\in I_{s|s}$ and $b\in I_{n|n}$. The superalgebra $\bar{\mathsf{A}}_{s,n}$ is also understood as the superalgebra presented by the generators $\bar{t}_{\alpha b}, \alpha\in I_{s|s}, b\in I_{n|n}$ and the relations:
\begin{eqnarray}
\bar{t}_{\alpha,b}&=&(-1)^{|\alpha|+|b|}\bar{t}_{-\alpha,-b},\label{eq:DQCA1}\\
\bar{T}^{1[3]}\bar{T}^{2[3]}S^{12}&=&S^{12}\bar{T}^{2[3]}\bar{T}^{1[3]},\label{eq:DQCA2}
\end{eqnarray}
where $\bar{T}=\sum\limits_{\alpha\in I_{s|s},b\in I_{n|n}}(-1)^{|b|(|\alpha|+|b|)} E_{b\alpha}\otimes \bar{t}_{\alpha b}$. 

The superalgebra $\bar{\mathsf{A}}_{s,n}$ is a $\Uq(\mathfrak{q}_s)\otimes\Uq(\mathfrak{q}_n)$-supermodule under the action $\Psi\otimes\Phi$. In terms of generating matrix, these actions can be written as:
\begin{equation}
L^{[2]3}\underset{\Phi}{\cdot}\bar{T}^{1[2]}=\left(S^{-1}\right)^{13}\bar{T}^{1[2]}\ \text{ and }\ L^{[2]3}\underset{\Psi}{\cdot}\bar{T}^{1[2]}=\bar{T}^{1[2]} \tilde{S}^{13},
\end{equation}
where $\tilde{S}=(1\otimes D)S(1\otimes D^{-1})$ and $D=\sum\limits_{\alpha\in I_{s|s}}q^{2(-1)^{|\alpha|}\alpha}E_{\alpha,\alpha}$. Analogous to the superalgebra $\mathsf{A}_{r,n}$, the superalgebra $\bar{\mathsf{A}}_{s,n}$ also admits the following multiplicity free decomposition
\begin{equation}\label{eq:barAdecom}
\bar{\mathsf{A}}_{s,n}=\bigoplus\limits_{\mu\in\mathrm{SP}\atop\ell(\mu)\leqslant\min(s,n)}\mathsf{L}_s(\mu)\circledast\mathsf{L}_n^*(\mu),
\end{equation}
where $\mathsf{L}_s(\mu)$ and $\mathsf{L}_s(\mu)\circledast\mathsf{L}_n^*(\mu)$ have the same meaning as in \eqref{eq:Adecom}.

\section{A braided tensor product}\label{braid}

In this section, we will create a quantization of the symmetric superalgebra $\mathrm{Sym}\left(V^{\oplus r}\oplus V^{*\oplus s}\right)$ for $V=\mathbb{C}^{n|n}$. Recall that $\mathsf{A}_{r,n}$ (resp. $\bar{\mathsf{A}}_{s,n}$) is the quantization of $\mathrm{Sym}\left(V^{\oplus r}\right)$ (resp. $\mathrm{Sym}\left(V^{*\oplus s}\right)$). A potential candidate for the quantization of $\mathrm{Sym}\left(V^{\oplus r}\oplus V^{*\oplus s}\right)$ is 
$$\mathcal{O}_{r,s}:=\mathsf{A}_{r,n}\otimes \bar{\mathsf{A}}_{s,n},$$ 
which is a $\Uq(\mathfrak{q}_r)\otimes\Uq(\mathfrak{q}_n)\otimes\Uq(\mathfrak{q}_s)$-supermodule under the action $\Psi\otimes\Phi\otimes\Psi$. In order to distinguish the actions of $\Uq(\mathfrak{q}_r)$ and $\Uq(\mathfrak{q}_s)$, we denote the action of $\Uq(\mathfrak{q}_s)$ by $\bar{\Psi}$. Explicitly,
\begin{align*}\label{eq:actions}
\Psi_{u'}(f\otimes\bar{g})=&\Psi_{u'}(f)\otimes\bar{g},&&u'\in\Uq(\mathfrak{q}_r),\\
\bar{\Psi}_{u''}(f\otimes\bar{g})=&(-1)^{|f||u''|}f\otimes\bar{\Psi}_{u''}(\bar{g}),&&u''\in\Uq(\mathfrak{q}_s), \\
\Phi_u(f\otimes\bar{g})=&\sum_{(u)}(-1)^{|f||u_{(2)}|}\Phi_{u_{(1)}}(f)\otimes\Phi_{u_{(2)}}(\bar{g}),&&u\in\Uq(\mathfrak{q}_n),
\end{align*}
where $f\in \mathsf{A}_{r,n}$,\ $\bar{g}\in\bar{\mathsf{A}}_{s,n}$ and $\Delta(u)=\sum\limits_{(u)}u_{(1)}\otimes u_{(2)}$.

The $\mathbb{C}(q)$-vector space $\mathcal{O}_{r,s}$ is naturally a superalgebra with the trivial multiplication:
$$(f\otimes \bar{g})(f'\otimes\bar{g}')=(-1)^{|f'||\bar{g}|}ff'\otimes \bar{g}\bar{g}',\qquad f,f'\in\mathsf{A}_{r,n}, \bar{g},\bar{g}'\in\bar{\mathsf{A}}_{s,n}.$$
However, the $\Uq(\mathfrak{q}_n)$-supermodule structure on $\mathcal{O}_{r,s}$ is not compatible with the trivial multiplication on $\mathcal{O}_{r,s}$ in the sense that $\mathcal{O}_{r,s}$ is not a $\Uq(\mathfrak{q}_n)$-supermodule superalgebra under the action $\Phi$. We have to modify the multiplication on the $\Uq(\mathfrak{q}_n)$-supermodule $\mathcal{O}_{r,s}$ such that it is compatible with the $\Uq(\mathfrak{q}_n)$-action $\Phi$, i.e, $\mathcal{O}_{r,s}$ is a $\Uq(\mathfrak{q}_n)$-supermodule superalgebra.

A natural way to create such a braided tensor product superalgebra is to use the intertwining operator (see \cite[Section~12]{CP94}), that is a homomorphism of $\Uq(\mathfrak{q}_n)$-supermodules $\Upsilon: \bar{\mathsf{A}}_{s,n}\otimes\mathsf{A}_{r,n}\rightarrow \mathsf{A}_{r,n}\otimes\bar{\mathsf{A}}_{s,n}$. We first recall a fundamental fact about intertwining operator.

\begin{lemma}\label{tensor}
Suppose that $U$ is a Hopf superalgebra, $\mathsf{A}$ and $\bar{\mathsf{A}}$ are two $U$-super- module superalgebras.
If\, $\Upsilon: \bar{\mathsf{A}}\otimes\mathsf{A}\rightarrow \mathsf{A}\otimes\bar{\mathsf{A}}$ is an intertwining operator such that the following diagrams commute:
\begin{equation}
\xymatrix{
\bar{\mathsf{A}}\otimes \mathsf{A}\otimes \mathsf{A}
\ar[d]_{1\otimes\mathrm{mul}} \ar[r]^{\Upsilon\otimes1}
&\mathsf{A}\otimes\bar{\mathsf{A}}\otimes \mathsf{A}
\ar[r]^{1\otimes\Upsilon}
&\mathsf{A}\otimes\mathsf{A}\otimes\bar{\mathsf{A}}
\ar[d]^{\mathrm{mul}\otimes1}\\
\bar{\mathsf{A}}\otimes\mathsf{A}
\ar[rr]_{\Upsilon}
&&\mathsf{A}\otimes\bar{\mathsf{A}}
}
\label{eq:intwass1}
\end{equation}
and
\begin{equation}
\xymatrix{
\bar{\mathsf{A}}\otimes \bar{\mathsf{A}}\otimes \mathsf{A}
\ar[r]^{1\otimes\Upsilon}\ar[d]_{\mathrm{mul}\otimes1}
&\bar{\mathsf{A}}\otimes\mathsf{A}\otimes \bar{\mathsf{A}}
\ar[r]^{\Upsilon\otimes1}
&\mathsf{A}\otimes\bar{\mathsf{A}}\otimes\bar{\mathsf{A}}
\ar[d]^{1\otimes\mathrm{mul}}\\
\bar{\mathsf{A}}\otimes\mathsf{A}
\ar[rr]_{\Upsilon}
&&\mathsf{A}\otimes\bar{\mathsf{A}}
}\label{eq:intwass2}
\end{equation}
where $\mathrm{mul}$ denotes the multiplication map, then the composition 
$$
\mathsf{A}\otimes\bar{\mathsf{A}}\otimes\mathsf{A}\otimes\bar{\mathsf{A}}\xrightarrow{\mathrm{id}\otimes\Upsilon\otimes\mathrm{id}}\mathsf{A}\otimes\mathsf{A}\otimes\bar{\mathsf{A}}\otimes\bar{\mathsf{A}}\xrightarrow{\mathrm{mul}\otimes\mathrm{mul}}\mathsf{A}\otimes\bar{\mathsf{A}}$$
is an associative multiplication on $\mathsf{A}\otimes\bar{\mathsf{A}}$, and $\mathsf{A}\otimes\bar{\mathsf{A}}$ is a $U$-supermodule superalgebra. \qed
\end{lemma}

If the Hopf algebra $U$ is quasi-triangular, then the composition of permutation operator and the canonical image of its universal $\mathcal{R}$-matrix gives an intertwining operator $\bar{\mathsf{A}}\otimes\mathsf{A}\rightarrow\mathsf{A}\otimes\bar{\mathsf{A}}$. We refer the reader to \cite[Theorem~2.3]{LZZ11} for the case of quasi-triangular quantum groups.  However, the quantum queer superalgebra $\Uq(\mathfrak{q}_n)$ is not quasi-triangular and a universal $\mathcal{R}$-matrix does not exist. Instead, we observe that the matrix $S$ in \eqref{eq:smatrix} determines an intertwining operator $\check{S}=P\circ S:V\otimes V\rightarrow V\otimes V$ that is the composition of the permutation operator with $S$. We will show that the matrix $S$ in \eqref{eq:smatrix} also determines an intertwining operator $\Upsilon:\bar{\mathsf{A}}_{s,n}\otimes\mathsf{A}_{r,n}\rightarrow\mathsf{A}_{r,n}\otimes\bar{\mathsf{A}}_{s,n}$. In order to properly define $\Upsilon$, we need the following equivalent presentation of $\mathsf{A}_{r,n}$.

\begin{lemma}
\label{lem:QCAP2}
An alternative presentation of the superalgebra $A_{r,n}$ is given by generators $1$, $t_{ia}$ with $i=1,\ldots,r,  a\in I_{n|n}$ and the relations
\begin{equation}
R^{12}T_+^{1[3]}T_+^{2[3]}=T_+^{2[3]}T_+^{1[3]}S^{12},\label{eq:ArlnsM}
\end{equation}
where $T_+=\sum\limits_{i=1}^{r}\sum\limits_{a\in I_{n|n}}E_{ia}\otimes t_{ia}$, $S$ is the matrix \eqref{eq:smatrix}, and 
\begin{equation}
R:=\sum\limits_{i,j=1}^rq^{\delta_{ij}}E_{ii}\otimes E_{jj}+\xi \sum\limits_{1\leqslant i<j\leqslant r }E_{ji}\otimes E_{ij}\,\label{eq:rmatrix}
\end{equation}
is the submatrix\footnote{The submatrix $R$ of $S$ is exactly the R-matrix of $\Uq(\mathfrak{gl}_r)$} of $S$ involving the terms $E_{ik}\otimes E_{jl}$ with $1\leqslant i, j,k,l\leqslant r$.
\end{lemma}

\begin{proof}
Recall that $\mathsf{A}_{r,n}$ is generated by $1$, $t_{i,a}$ with $i\in I_{r|r}$,  $a\in I_{n|n}$ and $t_{-i,-a}=t_{ia}$, we know that $\mathsf{A}_{r,n}$ is generated by $1, t_{ia}$ with $i=1,\ldots,r$ and $a\in I_{n|n}$. 

We next show that the relation \eqref{eq:QCA2} is isomorphic to \eqref{eq:ArlnsM} provided that $t_{-i,-a}=t_{ia}$. In fact, the relation \eqref{eq:QCA2} is equivalent to 
\begin{equation}
\begin{aligned}
&(-1)^{|i||j|+|j||b|+|b||i|}\left(q^{\varphi(i,j)}t_{ia}t_{jb}-(-1)^{(|i|+|a|)(|j|+|b|)}q^{\varphi(a,b)}t_{jb}t_{ia}\right)\\
=&\xi 
(\delta_{a<b}-\delta_{j<i})t_{ja}t_{ib}
+(-1)^{|j|+|b|}\xi(\delta_{-a<b}-\delta_{j<-i})t_{j,-a}t_{i,-b}.
\end{aligned}
\label{eq:QCA2E}
\end{equation}
for $i,j\in I_{r|r}, a,b\in I_{n|n}$. When $i,j>0$, it is reduced to
 \begin{equation}
\begin{aligned}
&q^{\delta_{ij}}t_{ia}t_{jb}-(-1)^{|a||b|}q^{\varphi(a,b)}t_{jb}t_{ia}\\
=&\xi (\delta_{a<b}-\delta_{j<i})t_{ja}t_{ib}
+(-1)^{|b|}\xi\delta_{-a<b}t_{j,-a}t_{i,-b},
\end{aligned}
\label{eq:QCAPE}
\end{equation}
which is the expansion of \eqref{eq:ArlnsM}.

Conversely, it suffices to show the relation \eqref{eq:QCAPE} also implies \eqref{eq:QCA2E}. We consider the following four cases separately:

\textbf{Case1:} $i>0, j>0$. In this situation, the relations \eqref{eq:QCAPE} and \eqref{eq:QCA2E} are same.

\textbf{Case2:} $i<0, j>0$. We deduce from (\ref{eq:QCAPE}) that
\begin{align*}
&q^{\delta_{-i,j}}t_{-i,-a}t_{jb}-(-1)^{|a||b|+|b|}q^{\varphi(-a,b)}t_{jb}t_{-i,-a}\\
=&\xi\left(\delta_{-a<b}-\delta_{j<-i}\right)t_{j,-a}t_{-i,b}
+(-1)^{|b|}\xi\delta_{a<b}t_{ja}t_{-i,-b}.
\end{align*}
Since $t_{-i,-a}=t_{ia}$, we obtain that
\begin{align*}
&q^{\delta_{-i,j}}t_{ia}t_{jb}-(-1)^{|a||b|+|b|}q^{\varphi(a,b)}t_{jb}t_{ia}\\
=&\xi\left(\delta_{-a<b}-\delta_{j<-i}\right)t_{j,-a}t_{i,-b}
+(-1)^{|b|}\xi\delta_{a<b}t_{ja}t_{ib},
\end{align*}
which coincides with (\ref{eq:QCA2E}) in the situation of $i<0$ and $j>0$.

\textbf{Case 3:} $i>0$, $j<0$. We deduce from (\ref{eq:QCAPE}) that
\begin{align*}
&q^{\delta_{i,-j}}t_{ia}t_{-j,-b}-(-1)^{|a||b|+|a|}q^{\varphi(a,-b)}t_{-j,-b}t_{ia}\\
=&\xi\left(\delta_{a<-b}-\delta_{-j<i}\right)t_{-j,a}t_{i,-b}
-(-1)^{|b|}\xi\delta_{b<a}t_{-j,-a}t_{ib}.
\end{align*}
Then $t_{-j,-a}=t_{ja}$ implies that 
\begin{align*}
&q^{\delta_{i,-j}}t_{ia}t_{jb}-(-1)^{|a||b|+|a|}q^{-\varphi(a,b)}t_{jb}t_{ia}\\
=&\xi\left(\delta_{a<-b}-\delta_{-j<i}\right)t_{j,-a}t_{i,-b}
-(-1)^{|b|}\xi\delta_{b<a}t_{ja}t_{ib}.
\end{align*}
Hence,
\begin{align*}
&q^{-\delta_{i,-j}}t_{ia}t_{jb}-(-1)^{|a|+|a||b|}q^{\varphi(a,b)}t_{jb}t_{ia}\\
=&\left(q^{\delta_{i,-j}}-\delta_{i,-j}\xi\right)t_{ia}t_{jb}\\
&-(-1)^{|a|+|a||b|}\left(q^{-\varphi(a,b)}+(-1)^{|b|}\left(\delta_{a,b}+\delta_{a,-b}\right)\xi\right)t_{jb}t_{ia}\\
=&\xi\left(\delta_{a<-b}-\delta_{-j<i}\right)t_{j,-a}t_{i,-b}
-(-1)^{|b|}\xi\delta_{b<a}t_{ja}t_{ib}-\delta_{i,-j}\xi t_{ia}t_{jb}\\
&-(-1)^{|a|+|b|+|a||b|}\left(\delta_{a,b}+\delta_{a,-b}\right)\xi t_{jb}t_{ia}\\
=&\xi\delta_{a\leqslant -b}t_{j,-a}t_{i,-b}-(-1)^{|b|}\xi\delta_{b\leqslant a}t_{ja}t_{ib}-\xi\delta_{-j\leqslant i}t_{j,-a}t_{i,-b}\\
=&(-1)^{|b|}\left(\delta_{a<b}-1\right)t_{ja}t_{ib}-\xi(\delta_{-a<b}-\delta_{j<i})t_{j,-a}t_{i,-b},
\end{align*}
which yields with (\ref{eq:QCA2E}) in the situation of $i>0$ and $j<0$.

\textbf{Case 4:} $i<0$, $j<0$. We deduce from (\ref{eq:QCAPE}) that
\begin{align*}
&q^{\delta_{ij}}t_{-i,-a}t_{-j,-b}-(-1)^{(|a|+1)(|b|+1)}q^{\varphi(-a,-b)}t_{-j,-b}t_{-i,-a}\\
=&\xi\left(\delta_{-a<-b}-\delta_{-j<-i}\right)t_{-j,-a}t_{-i,-b}
-(-1)^{|b|}\xi\delta_{a<-b}t_{-j,a}t_{-i,b}.
\end{align*}
It follows from $t_{-j,-a}=t_{j,a}$ that
\begin{align*}
&q^{\delta_{ij}}t_{ia}t_{jb}+(-1)^{|a||b|+|a|+|b|}q^{\varphi(a,-b)}t_{jb}t_{ia}\\
=&\xi \left(\delta_{b<a}-\delta_{i<j}\right)t_{ja}t_{ib}
-(-1)^{|b|}\xi\delta_{a<-b}t_{j,-a}t_{i,-b}.
\end{align*}
Hence,
\begin{align*}
&q^{-\delta_{ij}}t_{ia}t_{jb}+(-1)^{|a|+|b|+|a||b|}q^{\varphi(a,b)}t_{jb}t_{ia}\\
=&\left(q^{\delta_{ij}}-\delta_{ij}\xi\right)t_{ia}t_{jb}\\
&+(-1)^{|a|+|b|+|a||b|}
\left(q^{\varphi(a,-b)}-(-1)^{|b|}\left(\delta_{a,b}+\delta_{a,-b}\right)\xi\right)t_{jb}t_{ia}\\
=&\xi\left(\delta_{b<a}-\delta_{i<j}\right)t_{ja}t_{ib}-(-1)^{|b|}\xi\delta_{a<-b}t_{j,-a}t_{i,-b}\\
&-\xi\delta_{ij} t_{ia}t_{jb}+(-1)^{|a|+|a||b|}\xi\left(\delta_{a,b}+\delta_{a,-b}\right) t_{jb}t_{ia}\\
=&\xi\delta_{b\leqslant a}t_{ja}t_{ib}+(-1)^{|b|}\xi\delta_{a\leqslant -b}t_{j,-a}t_{i,-b}-\xi\delta_{i\leqslant j}t_{ja}t_{ib}\\
=&-\xi\left(\delta_{a<b}-\delta_{j<i}\right)t_{ja}t_{ib}+(-1)^{|b|}\xi\left(1-\delta_{-a<b}\right)t_{j,-a}t_{i,-b},
\end{align*}
which yields with \eqref{eq:QCA2E} in the situation of $i<0$ and $j<0$.

Now, we have shown that the relations \eqref{eq:QCA1} and \eqref{eq:QCA2} are equivalent to \eqref{eq:ArlnsM}. Hence, the generators $1, t_{ia}$, $i=1,\ldots, r$ and $a\in I_{n|n}$ and the relations \eqref{eq:ArlnsM} give an alternative presentation of the superalgebra $\mathsf{A}_{r,n}$.
\end{proof}

\begin{proposition}\label{multiplication}
There exist a $\mathbb{C}(q)$-linear map
$$\Upsilon: \bar{\mathsf{A}}_{s,n}\otimes\mathsf{A}_{r,n}\rightarrow\mathsf{A}_{r,n}\otimes\bar{\mathsf{A}}_{s,n}$$
such that the diagrams \eqref{eq:intwass1} and \eqref{eq:intwass2} commute,
\begin{equation}
\Upsilon(1\otimes f)=f\otimes 1, \qquad 
\Upsilon(\bar{g}\otimes1)=1\otimes\bar{g},
\label{eq:Updef1}
\end{equation}
for $f\in\mathsf{A}_{r,n}, \bar{g}\in\bar{\mathsf{A}}_{s,n}$, and 
\begin{equation}
\Upsilon\left(\bar{T}^{1[3]}T_+^{2[4]}\right)=T_+^{2[3]}\left(S^{-1}\right)^{12}\bar{T}^{1[4]}.\label{eq:Updef2}
\end{equation}
\end{proposition}
\begin{proof}
Since $\mathsf{A}_{r,n}$ is generated by $1$ and $t_{ia}$ with $i=1,\ldots, r$, $a\in I_{n|n}$, while $\bar{\mathsf{A}}_{s,n}$ is generated by $1$ and $\bar{t}_{\alpha, b}$ with $\alpha\in I_{s|s}, b\in I_{n|n}$, the equations \eqref{eq:Updef1} and \eqref{eq:Updef2} determine $\Upsilon(\bar{g}\otimes f)$ when $f$ and $\bar{g}$ are one of those generators of $\mathsf{A}_{r,n}$ and $\bar{\mathsf{A}}_{s,n}$, respectively. The definition is then extended to the whole $\bar{\mathsf{A}}_{s,n}\otimes\mathsf{A}_{r,n}$ according to diagrams \eqref{eq:intwass1} and \eqref{eq:intwass2}. It suffices to show $\Upsilon$ is well-defined.

By Lemma~\ref{lem:QCAP2}, the generators $t_{ia}, i=1,\ldots, r, a\in I_{n|n}$ of $\mathsf{A}_{r,n}$ satisfy
$$R^{23}T_+^{2[5]}T_+^{3[5]}=T_+^{3[5]}T_+^{2[5]}S^{23}.$$
We calculate using the diagram \eqref{eq:intwass1} that
\begin{align*}
\Upsilon\left(\bar{T}^{1[4]}R^{23}T_+^{2[5]}T_+^{3[5]}\right)
=&R^{23}\Upsilon(1\otimes\mathrm{mul})\left(\bar{T}^{1[4]}T_+^{2[5]}T_+^{3[5]}\right)\\
=&R^{23}(\mathrm{mul}\otimes 1)(1\otimes\Upsilon)(\Upsilon\otimes1)\left(\bar{T}^{1[4]}T_+^{2[5]}T_+^{3[6]}\right)\\
=&R^{23}(\mathrm{mul}\otimes 1)(1\otimes\Upsilon)\left(T_+^{2[4]}\left(S^{-1}\right)^{12}\bar{T}^{1[5]}T_+^{3[6]}\right)\\
=&R^{23}(\mathrm{mul}\otimes 1)\left(T_+^{2[4]}\left(S^{-1}\right)^{12}T_+^{3[5]}\left(S^{-1}\right)^{13}\bar{T}^{1[6]}\right)\\
=&R^{23}T_+^{2[4]}T_+^{3[4]}\left(S^{-1}\right)^{12}\left(S^{-1}\right)^{13}\bar{T}^{1[5]}\\
=&T_+^{3[4]}T_+^{2[4]}S^{23}\left(S^{-1}\right)^{12}\left(S^{-1}\right)^{13}\bar{T}^{1[5]}.\\
\Upsilon\left(\bar{T}^{1[4]}T_+^{3[5]}T_+^{2[5]}S^{23}\right)
=&\Upsilon\circ(1\otimes\mathrm{mul})\left(\bar{T}^{1[4]}T_+^{3[5]}T_+^{2[6]}S^{23}\right)\\
=&(\mathrm{mul}\otimes 1)(1\otimes\Upsilon)(\Upsilon\otimes1)\left(\bar{T}^{1[4]}T_+^{3[5]}T_+^{2[6]}S^{23}\right)\\
=&(\mathrm{mul}\otimes 1)(1\otimes\Upsilon)\left(T_+^{3[4]}\left(S^{-1}\right)^{13}\bar{T}^{1[5]}T_+^{2[6]}S^{23}\right)\\
=&(\mathrm{mul}\otimes 1)\left(T_+^{3[4]}\left(S^{-1}\right)^{13}T_+^{2[5]}\left(S^{-1}\right)^{12}\bar{T}^{1[6]}S^{23}\right)\\
=&T_+^{3[4]}T_+^{2[4]}\left(S^{-1}\right)^{13}\left(S^{-1}\right)^{12}S^{23}\bar{T}^{1[5]}.
\end{align*}

Since the matrix $S$ satisfies the quantum Yang-Baxter equation, we conclude that
\begin{equation}
\Upsilon\left(\bar{T}^{1[4]}R^{23}T_+^{2[5]}T_+^{3[5]}\right)=\Upsilon\left(\bar{T}^{1[4]}T_+^{3[5]}T_+^{2[5]}S^{23}\right).
\label{eq:Upsilon1}
\end{equation}
\medskip

The generators $\bar{t}_{\alpha b}$'s of $\bar{\mathsf{A}}_{s,n}$ satisfy
$$\bar{T}^{1[4]}\bar{T}^{2[4]}S^{12}=S^{12}\bar{T}^{2[4]}\bar{T}^{1[4]}.$$
We similarly verify that
\begin{equation}
\Upsilon\left(\bar{T}^{1[4]}\bar{T}^{2[4]}S^{12}T_+^{3[5]}\right)
=\Upsilon\left(S^{12}\bar{T}^{2[4]}\bar{T}^{1[4]}T_+^{3[5]}\right).
\label{eq:Upsilon2}
\end{equation}

Moreover, $\bar{t}_{-\alpha,-b}=(-1)^{|\alpha|+|b|}\bar{t}_{\alpha, b}$ for $\alpha\in I_{s|s}$ and $b\in I_{n|n}$. It can be written in the matrix form $\bar{T}=(J_n\otimes 1)\bar{T}(J_s\otimes1)$, where
\begin{equation}
J_n=\sum_{a\in I_{n|n}}(-1)^{|a|}E_{-a,a}. \label{eq:Jop}
\end{equation}

Since $(J_n\otimes 1)S=S(J_n\otimes1)$, we also verify that
\begin{align*}
(J_n\otimes1\otimes1\otimes1)T_+^{2[3]}\left(S^{-1}\right)^{12}\bar{T}^{1[4]}
=&T_+^{2[3]}\left(S^{-1}\right)^{12}(J_n\otimes1\otimes1\otimes1)\bar{T}^{1[4]}\\
=&-T_+^{2[3]}\left(S^{-1}\right)^{12}\bar{T}^{1[4]}(J_s\otimes1\otimes1\otimes1).
\end{align*}
which yields that 
\begin{equation}
\Upsilon\left(\left(J_n\otimes1\otimes1\otimes1\right)\bar{T}^{1[3]}\left(J_s\otimes1\otimes1\otimes1\right)T_+^{2[4]}\right)
=\Upsilon\left(\bar{T}^{1[3]}T_+^{2[4]}\right).\label{eq:Upsilon3}
\end{equation}

Now, the well-definedness of $\Upsilon$ follows from \eqref{eq:Upsilon1}, \eqref{eq:Upsilon2} and \eqref{eq:Upsilon3}.
\end{proof}

\begin{proposition}
\label{prop:Upqn}
The $\mathbb{C}(q)$-linear map $\Upsilon: \bar{\mathsf{A}}_{s,n}\otimes\mathsf{A}_{r,n}\rightarrow\mathsf{A}_{r,n}\otimes\bar{\mathsf{A}}_{s,n}$ is a homomorphism of $\Uq(\mathfrak{q}_n)$-supermodules.
\end{proposition}
\begin{proof}
For $u\in\Uq(\mathfrak{q}_n)$ and $f\in\mathsf{A}_{r,n}$, if $\Delta(u)=\sum\limits_{(u)}u_{(1)}\otimes u_{(2)}$, then
\begin{align*}
\sum\limits_{(u)}\Upsilon\circ\left(\Phi_{u_{(1)}}\otimes\Phi_{u_{(2)}}\right)(1\otimes f)
=&\sum\limits_{(u)}\Upsilon\left(\varepsilon(u_{(1)})\otimes\Phi_{u_{(2)}}(f)\right)
=\Upsilon\left(1\otimes \Phi_u(f)\right)
=\Phi_u(f)\otimes1,\\
\sum\limits_{(u)}\left(\Phi_{u_{(1)}}\otimes\Phi_{u_{(2)}}\right)\circ\Upsilon(1\otimes f)
=&\sum\limits_{(u)}\Phi_{u_{(1)}}(f)\otimes\Phi_{u_{(2)}}(1)
=\sum\limits_{(u)}\Phi_{u_{(1)}}(f)\otimes\varepsilon(u_{(2)})
=\Phi_u(f)\otimes1,
\end{align*}
which yields
$$\sum\limits_{(u)}\Upsilon\circ\left(\Phi_{u_{(1)}}\otimes\Phi_{u_{(2)}}\right)(1\otimes f)
=\sum\limits_{(u)}\left(\Phi_{u_{(1)}}\otimes\Phi_{u_{(2)}}\right)\circ\Upsilon(1\otimes f)$$
for $f\in\mathsf{A}_{r,n}$. A similar calculation also shows
$$\sum\limits_{(u)}\Upsilon\circ\left(\Phi_{u_{(1)}}\otimes\Phi_{u_{(2)}}\right)(\bar{g}\otimes 1)
=\sum\limits_{(u)}\left(\Phi_{u_{(1)}}\otimes\Phi_{u_{(2)}}\right)\circ\Upsilon(\bar{g}\otimes 1)$$
for $\bar{g}\in\bar{\mathsf{A}}_{s,n}$.

In order to show that
$$\sum_{(u)}\Upsilon\circ\left(\Phi_{u_{(1)}}\otimes\Phi_{u_{(2)}}\right)(\bar{t}_{\alpha b}\otimes t_{ka})
=\sum_{(u)}\left(\Phi_{u_{(1)}}\otimes\Phi_{u_{(2)}}\right)\circ\Upsilon(\bar{t}_{\alpha b}\otimes t_{ka}),$$
for $a,b\in I_{n|n}$, $\alpha\in I_{s|s}$ and $k=1,\ldots, r$, it suffices to consider the case where $u=L_{ij}$ with $i\leqslant j$. Then the verification can be done using matrix computations as follows:
\begin{align*}
\Upsilon\left(L\underset{\Phi}{\cdot}\bar{T}^{1[3]}T_+^{2[4]}\right)
=&\Upsilon\left(\left(L^{[3]5}\underset{\Phi}{\cdot}\bar{T}^{1[3]}\right)\left(L^{[4]5}\underset{\Phi}{\cdot}T_+^{2[4]}\right)\right)\\
=&\Upsilon\left(\left(S^{-1}\right)^{15}\bar{T}^{1[3]}T_+^{2[4]}S^{25}\right)\\
=&\left(S^{-1}\right)^{15}T_+^{2[3]}\left(S^{-1}\right)^{12}\bar{T}^{1[4]}S^{25}\\
=&T_+^{2[3]}\left(S^{-1}\right)^{15}\left(S^{-1}\right)^{12}S^{25}\bar{T}^{1[4]}.\\
L\underset{\Phi}{\cdot}\Upsilon\left(\bar{T}^{1[3]}T_+^{2[4]}\right)
=&L\underset{\Phi}{\cdot}\left(T_+^{2[3]}\left(S^{-1}\right)^{12}\bar{T}^{1[4]}\right)\\
=&\left(L^{[3]5}\underset{\Phi}{\cdot}T_+^{2[3]}\right)\left(S^{-1}\right)^{12}\left(L^{[4]5}\underset{\Phi}{\cdot}\bar{T}^{1[4]}\right)\\
=&T_+^{2[3]}S^{25}\left(S^{-1}\right)^{12}\left(S^{-1}\right)^{15}\bar{T}^{1[4]}.
\end{align*}
Using the quantum Yang-Baxter equation for $S$ again, we obtain
$$\Upsilon\left(L\underset{\Phi}{\cdot}\bar{T}^{1[3]}T_+^{2[4]}\right)
=L\underset{\Phi}{\cdot}\Upsilon\left(\bar{T}^{1[3]}T_+^{2[4]}\right).$$

Now, $\mathsf{A}_{r,n}$ is generated by $1, t_{ka}$, $k=1,\ldots, r$ and $a\in I_{n|n}$, while $\bar{\mathsf{A}}_{s,n}$ is generated by $1, \bar{t}_{\alpha,b}$, $\alpha\in I_{s|s}$ and $b\in I_{n|n}$. Using the commutative diagrams \eqref{eq:intwass1} and \eqref{eq:intwass2}, we inductively deduce that 
$$\sum_{(u)}\Upsilon\circ\left(\Phi_{u_{(1)}}\otimes\Phi_{u_{(2)}}\right)(\bar{g}\otimes f)
=\sum_{(u)}\left(\Phi_{u_{(1)}}\otimes\Phi_{u_{(2)}}\right)\circ\Upsilon(f\otimes \bar{g}),$$
for all  $f\in\mathsf{A}_{r,n}$ and $\bar{g}\in\bar{\mathsf{A}}_{s,n}$. This shows $\Upsilon$ is a homomorphism of $\Uq(\mathfrak{q}_n)$-supermodules.
\end{proof}

\begin{proposition}
\label{prop:Upqs}
The $\mathbb{C}(q)$-linear map $\Upsilon: \bar{\mathsf{A}}_{s,n}\otimes\mathsf{A}_{r,n}\rightarrow\mathsf{A}_{r,n}\otimes\bar{\mathsf{A}}_{s,n}$ is a homomorphism of $\Uq(\mathfrak{q}_s)$-supermodules under the action $\bar{\Psi}$.
\end{proposition}
\begin{proof}
It suffices to verify
$$\Upsilon\circ\left(\bar{\Psi}_u\otimes1\right)=\left(1\otimes\bar{\Psi}_u\right)\circ\Upsilon,\qquad u\in\Uq(\mathfrak{q}_s).$$
We first check it on generators.
\begin{align*}
\Upsilon\circ(\bar{\Psi}_u\otimes1)(\bar{g}\otimes 1)
=&\Upsilon\left(\bar{\Psi}_u(\bar{g})\otimes1\right)
=1\otimes\bar{\Psi}_u(\bar{g})=(1\otimes\bar{\Psi}_u)\circ\Upsilon(\bar{g}\otimes1),\\
\Upsilon\circ(\bar{\Psi}_u\otimes1)(1\otimes f)
=&\Upsilon\left(\varepsilon(u)\otimes f\right)
=f\otimes\varepsilon(u)=(1\otimes\bar{\Psi}_u)\circ\Upsilon(1\otimes f),
\end{align*}
for $f\in\mathsf{A}_{r,n}$ and $\bar{g}\in\bar{\mathsf{A}}_{s,n}$. We also have
\begin{align*}
\Upsilon\left(L^{[3]5}\underset{\Psi}{\cdot}(\bar{T}^{1[3]}T_+^{2[4]})\right)
=&\Upsilon\left(\bar{T}^{1[3]}\tilde{S}^{13}T_+^{2[4]}\right)
=\Upsilon\left(\bar{T}^{1[3]}T_+^{2[4]}\right)\tilde{S}^{13}\\
=&T_+^{2[3]}\left(S^{-1}\right)^{12}\bar{T}^{1[4]}\tilde{S}^{13},\\
L^{[4]5}\underset{\Psi}{\cdot}\Upsilon\left(\bar{T}^{1[3]}T_+^{2[4]}\right)
=&L^{[4]5}\underset{\Psi}{\cdot}\left(T_+^{2[3]}\left(S^{-1}\right)^{12}\bar{T}^{1[4]}\right)
=T_+^{2[3]}\left(S^{-1}\right)^{12}\bar{T}^{1[4]}\tilde{S}^{13},
\end{align*}
which implies
$$\Upsilon\circ(\bar{\Psi}_u\otimes1)(\bar{t}_{\alpha b}\otimes t_{ka})
=(1\otimes\bar{\Psi}_u)\circ\Upsilon(\bar{t}_{\alpha b}\otimes t_{ka}),$$
for $k=1,\ldots,r$, $\alpha\in I_{s|s}$ and $a,b\in I_{n|n}$. Using the commutative diagrams \eqref{eq:intwass1} and \eqref{eq:intwass2}, we deduce the following equation by induction
$$\Upsilon\circ(\bar{\Psi}_u\otimes1)(\bar{g}\otimes f)
=(1\otimes\bar{\Psi}_u)\circ\Upsilon(\bar{g}\otimes f),$$
for $f\in\mathsf{A}_{r,n}$ and $\bar{g}\in\bar{\mathsf{A}}_{s,n}$. This completes the proof.
\end{proof}

\begin{remark} 
\label{rmk:Upqr}
We have shown in Propositions~\ref{prop:Upqn} and~\ref{prop:Upqs} that $\Upsilon$ is a homomorphism of $\Uq(\mathfrak{q}_n)$-supermodules as well as $\Uq(\mathfrak{q}_s)$-supermodules. But $\Upsilon$ is not a homomorphism of $\Uq(\mathfrak{q}_r)$-supermodules under the action $\Psi$. In fact, it is known from \eqref{eq:Updef2} that
\begin{align*}
\Upsilon(\bar{t}_{\alpha b}\otimes t_{ia})=&(-1)^{|a|(|\alpha|+|b|)}q^{-\varphi(b,a)}t_{ia}\otimes \bar{t}_{\alpha b}\\
&-(-1)^{|a||\alpha|}\xi\sum\limits_{p<a}t_{ip}\otimes \left(\delta_{ab}\bar{t}_{\alpha p}+(-1)^{|p|}\delta_{a,-b}\bar{t}_{\alpha,-p}\right),
\end{align*}
for $k=1,\ldots r$, $\alpha\in I_{s|s}$ and $a,b\in I_{n|n}$. For $i<0<j$, the action of $L_{ij}\in\Uq(\mathfrak{q}_r)$ on $\mathrm{A}_{r,n}$ is given by 
$$\Psi_{L_{ij}}(t_{ka})=(-1)^{|a|}\delta_{jk}t_{-i,-a},\quad k=1,\ldots, r, \quad a\in I_{n|n}.$$
Then we directly compute that
\begin{align*}
&\Upsilon\circ(1\otimes\Psi_{L_{ij}})(\bar{t}_{\alpha b}\otimes t_{ka})-(\Psi_{L_{ij}}\otimes1)\circ\Upsilon(\bar{t}_{\alpha b}\otimes t_{ka})\\
=&(-1)^{|a|(|\alpha|+|b|)}\xi^2\delta_{jk}\sum\limits_{p\in I_{n|n}}t_{-i,-p}\otimes\left(\delta_{a,-b}\bar{t}_{\alpha,-p}+\delta_{ab}(-1)^{|p|}\bar{t}_{\alpha p}\right),
\end{align*} 
for $i<0<j$, which shows that $\Upsilon\circ(1\otimes\Psi_{L_{ij}})$ is not necessarily equal to $(\Psi_{L_{ij}}\otimes1)\circ\Upsilon$.
\end{remark}

Combining with Lemma \ref{tensor} and Proposition \ref{multiplication}, we obtain a braided multiplication on 
$\mathcal{O}_{r,s}=\mathsf{A}_{r,n}\otimes \bar{\mathsf{A}}_{s,n}$. We identify $t_{ia}\in\mathsf{A}_{r,n}$ with $t_{ia}\otimes 1\in \mathcal{O}_{r,s}$, and  $\bar{t}_{\alpha b}\in\bar{\mathsf{A}}_{s,n}$ with $1\otimes \bar{t}_{\alpha b}\in \mathcal{O}_{r,s}$. Under the new multiplication,
\begin{equation}
\begin{aligned}
\bar{t}_{\alpha b}t_{ia}=&(-1)^{|a|(|\alpha|+|b|)}q^{-\varphi(b,a)}t_{ia}\bar{t}_{\alpha b}\\
&-(-1)^{|a||\alpha|}\xi \sum\limits_{p<a}t_{ip}\left(\delta_{ab}\bar{t}_{\alpha p}+\delta_{a,-b}(-1)^{|p|}\bar{t}_{\alpha,-p}\right),
\end{aligned}
\label{eq:braidmul1}
\end{equation}
for $i=1,\ldots, r$, $a,b\in I_{n|n}$ and $\alpha\in I_{s|s}$. It is also written in the following matrix form:
$$\bar{T}^{1[3']}T_+^{2[3]}=T_+^{2[3]}\left(S^{-1}\right)^{12}\bar{T}^{1[3']}.$$
We obtain an associative superalgebra $\mathcal{O}_{r,s}$, which could be viewed as the quantization of  the symmetric superalgebra $\mathrm{Sym}(V^{\oplus r}\oplus V^{*\oplus s})$. 

By Propositions~\ref{prop:Upqn} and~\ref{prop:Upqs}, the multiplication on $\mathcal{O}_{r,s}$ is compatible with the $\Uq(\mathfrak{q}_n)$-action $\Phi$ and $\Uq(\mathfrak{q}_s)$-action $\bar{\Psi}$, i.e., $\mathcal{O}_{r,s}$ is a $\Uq(\mathfrak{q}_n)$-supermodule superalgebra under the action $\Phi$ and a $\Uq(\mathfrak{q}_s)^{\mathrm{cop}}$-supermodule superalgebra under the action $\bar{\Psi}$. But we observe from Remark~\ref{rmk:Upqr} that $\mathcal{O}_{r,s}$ is not a $\Uq(\mathfrak{q}_r)^{\mathrm{cop}}$-supermodule superalgebra under the action $\Psi$. We will investigate the $\mathrm{U}_q(\mathfrak{q}_n)$-invariant sub-superalgebra of  $\mathcal{O}_{r,s}$.

\begin{remark}
In the definition of the intertwining operator $\Upsilon$ in Proposition~\ref{multiplication}, we employ the presentation of $\mathsf{A}_{r,n}$ given in Lemma \ref{lem:QCAP2} that uses the generators $t_{ia}$ with $i>0$. 
 It seems that a different presentation of $\mathsf{A}_{r,n}$ using half of generators (choose either $t_{ia}$ or $t_{-i,a}$ for each $i=1,\ldots, r$ and $a\in I_{n|n}$) yields a different intertwining operator $\Upsilon'$. For example, we may choose  generators $1, t_{-i,a}$ for all $i=1,\ldots,r$ and $a\in I_{n|n}$, then $\mathsf{A}_{r,n}$ is presented by the relation
\begin{equation}
R^{\prime 12}T_-^{1[3]}T_-^{2[3]}=T_-^{2[3]}T_-^{1[3]}S^{\prime 12},
\end{equation}
where 
$T_-=\sum\limits_{i=1}^{r}\sum\limits_{a\in I_{n|n}}E_{-i,a}\otimes t_{-i,a}$,  $S^{\prime}=(1\otimes J_n)S(1\otimes J_n)$ and $R^{\prime}$ is the submatix of $S^{\prime}$ involving the terms $E_{ik}\otimes E_{jl}$ for $-r\leqslant i,j,k,l\leqslant -1$. Then we can define another intertwining operator $\Upsilon'$ as in Proposition \ref{multiplication} such that 
$$\Upsilon'\left(\bar{T}^{1[3]}T_-^{2[4]}\right)=T_-^{2[3]}\left(S^{\prime -1}\right)^{12}\bar{T}^{1[4]}.$$
One directly verifies $\Upsilon'$ is equal to $\Upsilon$. But the uniqueness of an intertwining operator $\bar{\mathsf{A}}_{s,n}\otimes\mathsf{A}_{r,n}\rightarrow\mathsf{A}_{r,n}\otimes\bar{\mathsf{A}}_{s,n}$ satisfying Lemma~\ref{tensor} is unknown yet.
\end{remark}

\section{Invariants}
\label{Invariant}
In the previous section, we created the intertwining operator $$\Upsilon:\bar{\mathsf{A}}_{s,n}\otimes\mathsf{A}_{r,n}\rightarrow\mathsf{A}_{r,n}\otimes\bar{\mathsf{A}}_{s,n}$$ 
and thus obtained a braided tensor product superalgebra $\mathcal{O}_{r,s}=\mathsf{A}_{r,n}\otimes\bar{\mathsf{A}}_{s,n}$, which is a quantization of the symmetric superalgebra $\mathrm{Sym}(V^{\oplus r}\oplus V^{*\oplus s})$. Since $\mathcal{O}_{r,s}$ is a $\Uq(\mathfrak{q}_n)$-supermodule superalgebra, the $\Uq(\mathfrak{q}_n)$-invariants 
$$\left(\mathcal{O}_{r,s}\right)^{\mathrm{U}_q(\mathfrak{q}_n)}=\left\{z\in\mathcal{O}_{r,s}|\ \Phi_u(z)=\varepsilon(u)z,\ \forall u\in\mathrm{U}_q(\mathfrak{q}_n)\right\}$$
is a sub-superalgebra of $\mathcal{O}_{r,s}$. We aim to provide an explicit presentation of the invariant sub-superalgebra $\left(\mathcal{O}_{r,s}\right)^{\Uq(\mathfrak{q})}$ in terms of generators and relations. Such a description of generators and relations are usually called the first and second fundamental theorem (FFT and SFT) in the invariant theory of $\Uq(\mathfrak{q}_n)$, respectively. We will prove the FFT in Section~\ref{FFT} and discuss the SFT in a separate paper. We find a few invariant elements and discuss the properties of the invariant sub-superalgebra $\left(\mathcal{O}_{r,s}\right)^{\Uq(\mathfrak{q}_n)}$ in this section.

\begin{lemma}
For $i=1,\ldots, r$ and $\alpha\in I_{s|s}$, the element
\begin{align*}
x_{i\alpha}=&\sum\limits_{p\in I_{n|n}}t_{ip}\otimes\bar{t}_{\alpha p},
\end{align*}
are contained in $\left(\mathcal{O}_{r,s}\right)^{\mathrm{U}_q(\mathfrak{q}_n)}$.
\end{lemma}

\begin{proof}
We set
\begin{equation}
X=\sum\limits_{i=1}^r\sum\limits_{\alpha\in I_{s|s}}E_{i\alpha}\otimes x_{i\alpha}.
\end{equation}
Then $X^{1[2]}=T_+^{1[2]}\bar{T}^{1[2']}$. Since $\Delta(L^{[2]3})=L^{2[3]}\otimes L^{[2']3}$, the action of generating matrix $L$ on $X$ under $\Phi$ is 
\begin{align*}
L^{[2]3}\underset{\Phi}{\cdot}X^{1[2]}
=&\left(L^{[2]3}\underset{\Phi}{\cdot}T_+^{1[2]}\right)\left(L^{[2']3}\underset{\Phi}{\cdot}\bar{T}^{1[2']}\right)
=T_+^{1[2]}S^{13}\left(S^{-1}\right)^{13}\bar{T}^{1[2']}\\
=&T_+^{1[2]}\bar{T}^{1[2']}
=X^{1[2]}.
\end{align*}
Since the counit $\varepsilon$ of $\Uq(\mathfrak{q}_n)$ satisfies $\varepsilon(L)=I$, we obtain $x_{i\alpha}\in\left(\mathcal{O}_{r,s}\right)^{\Uq(\mathfrak{q}_n)}$.
\end{proof}

\begin{proposition}
In the superalgebra $\mathcal{O}_{r,s}$, the elements $x_{i\alpha}$, $t_{ja}$ and $\bar{t}_{\alpha b}$'s satisfy the following relations:
\begin{align}
R^{12}X^{1[3]}T_+^{2[3]}=&T_+^{2[3]}X^{1[3]},\label{eq:XT}\\
\bar{T}^{1[3']}X^{2[3]}S^{12}=&X^{2[3]}\bar{T}^{1[3']},\label{eq:XbarT}
\end{align}
where $R$ and $S$ are given in \eqref{eq:rmatrix} and \eqref{eq:smatrix} respectively.
\end{proposition}
\begin{proof}
Note that $X^{1[3]}=T_+^{1[3]}\bar{T}^{1[3']}$. We have already known from Lemma~\ref{lem:QCAP2} and Theorem~\ref{multiplication} that $$R^{12}T_+^{1[3]}T_+^{2[3]}=T_+^{2[3]}T_+^{1[3]}S^{12},\text{ and } \bar{T}^{1[3]}T_+^{1[2]}=T_+^{1[2]}\left(S^{-1}\right)^{12}\bar{T}^{1[3]}$$ 
in $\mathcal{O}_{r,s}$. Hence, we deduce that
\begin{align*}
R^{12}X^{1[3]}T_+^{2[3]}=&R^{12}T_+^{1[3]}\bar{T}^{1[3']}T_+^{2[3]}=R^{12}T_+^{1[3]}T_+^{2[3]}\left(S^{-1}\right)^{12}\bar{T}^{1[3']}\\
=&T_+^{2[3]}T_+^{1[3]}S^{12}\left(S^{-1}\right)^{12}\bar{T}^{1[3']}=T_+^{2[3]}X^{1[3]}.
\end{align*}
Similarly, 
\begin{align*}
\bar{T}^{1[3']}X^{2[3]}S^{12}=&\bar{T}^{1[3']}T_+^{2[3]}\bar{T}^{2[3']}S^{12}=T_+^{2[3]}\left(S^{-1}\right)^{12}\bar{T}^{1[3']}\bar{T}^{2[3']}S^{12}\\
=&T_+^{2[3]}\bar{T}^{2[3']}\bar{T}^{1[3']}\left(S^{-1}\right)^{12}S^{12}=X^{2[3]}\bar{T}^{1[3']}.
\end{align*}
This completes the proof.
\end{proof}

\begin{proposition}
\label{prop:XX}
In the superalgebra $\mathcal{O}_{r,s}$, the elements $x_{i\alpha}$'s satisfy the following relations:
\begin{align}
R^{12}X^{1[3]}X^{2[3]}=X^{2[3]}X^{1[3]}\left(S^{T}\right)^{12},\label{eq:XX}
\end{align}
where $S^T=PSP$ and $P=\sum\limits_{\alpha,\beta\in I_{s|s}}(-1)^{|\alpha|}E_{\beta\alpha}\otimes E_{\alpha\beta}$ is the permutation operator.
\end{proposition}
\begin{proof}
Since $X^{2[3]}=T_+^{2[3]}\bar{T}^{2[3']}$, it follows from \eqref{eq:XT} that
$$R^{12}X^{1[3]}X^{2[3]}
=R^{12}X^{1[3]}T_+^{2[3]}\bar{T}^{2[3']}
=T_+^{2[3]}X^{1[3]}\bar{T}^{2[3']}.$$
By \eqref{eq:XbarT}, we have
$$X^{1[3]}\bar{T}^{2[3']}=P^{12}X^{2[3]}\bar{T}^{1[3']}P^{12}\\
=P^{12}\bar{T}^{1[3']}X^{2[3]}S^{12}P^{12}\\
=\bar{T}^{2[3']}X^{1[3]}P^{12}S^{12}P^{12}.$$
This implies \eqref{eq:XX}.
\end{proof}

In the superalgebra $\mathcal{O}_{r,s}$, we have found a family of $\Uq(\mathfrak{q}_n)$-invariant elements $x_{i\alpha}$, $i=1,\ldots, r$ and $\alpha\in I_{s|s}$, which will be demonstrated to generate the whole invariant sub-super algebra $\left(\mathcal{O}_{r,s}\right)^{\Uq(\mathfrak{q}_n)}$ in the next section. Besides a $\Uq(\mathfrak{q}_n)$-supermodule, $\mathcal{O}_{r,s}$ is also a $\Uq(\mathfrak{q}_r)$-supermodule under the action $\Psi$ and a $\Uq(\mathfrak{q}_s)$-supermodule under the action $\bar{\Psi}$. Both the actions $\Psi$ and $\bar{\Psi}$ commutes with the action $\Phi$ of $\Uq(\mathfrak{q}_n)$. Then the Howe dualities lead to the following decomposition of the invariant sub-superalgebra $\left(\mathcal{O}_{r,s}\right)^{\Uq(\mathfrak{q}_n)}$ as a $\Uq(\mathfrak{q}_r)\otimes\Uq(\mathfrak{q}_s)$-supermodule.

\begin{theorem}
\label{thm:InvDec}
The invariant sub-superalgebra $\left(\mathcal{O}_{r,s}\right)^{\Uq(\mathfrak{q}_n)}$ is a $\Uq(\mathfrak{q}_r)\otimes\Uq(\mathfrak{q}_s)$-supermodule under the action $\Psi\otimes\bar{\Psi}$. Moreover, it admits a multiplicity free decomposition
\begin{equation}
\left(\mathcal{O}_{r,s}\right)^{\Uq(\mathfrak{q}_n)}=\bigoplus_{\lambda\in\mathrm{SP}\atop\ell(\lambda)\leqslant\min(r,s,n)}\left(\mathsf{L}_r^*(\lambda)\circledast\mathsf{L}_s(\lambda)\right).\label{eq:InvDec}
\end{equation}
\end{theorem}
\begin{proof}
Since the superalgebra $\mathsf{A}_{r,n}$ (resp. $\bar{\mathsf{A}}_{s,n}$) is a $\Uq(\mathfrak{q}_r)\otimes\Uq(\mathfrak{q}_n)$-supermodule (resp. a $\Uq(\mathfrak{q}_s)\otimes\Uq(\mathfrak{q}_n)$-supermodule) and has the multiplicity free decomposition \eqref{eq:Adecom} (resp. \eqref{eq:barAdecom}), the superalgebra $\mathcal{O}_{r,s}=\mathsf{A}_{r,n}\otimes\bar{\mathsf{A}}_{s,n}$ is decomposed into the direct sum
$$\mathcal{O}_{r,s}=\bigoplus_{\lambda,\mu\in\mathrm{SP}\atop \ell(\lambda)\leqslant\min(r,n), \ell(\mu)\leqslant\min(s,n)}\left(\mathsf{L}_r^*(\lambda)\circledast\mathsf{L}_n(\lambda)\right)\otimes\left(\mathsf{L}_s(\mu)\circledast\mathsf{L}_n^*(\mu)\right),$$
as a $\Uq(\mathfrak{q}_r)\otimes\Uq(\mathfrak{q}_n)\otimes\Uq(\mathfrak{q}_s)\otimes\Uq(\mathfrak{q}_n)$-supermodule.

We denote
$$\delta(\lambda)=\begin{cases}0,&\text{if }\ell(\lambda)\text{ is even},\\ 1,&\text{if }\ell(\lambda)\text{ is odd}.\end{cases}$$
Then Schur's Lemma for queer superalgebra states that
$$\dim\left(\mathsf{L}_n(\lambda)\otimes\mathsf{L}_n^*(\mu)\right)^{\Uq(\mathfrak{q}_n)}=\delta_{\lambda,\mu}2^{\delta(\lambda)},$$
which implies that
$$\left(\left(\mathsf{L}_r^*(\lambda)\otimes\mathsf{L}_n(\lambda)\right)\otimes\left(\mathsf{L}_s(\mu)\otimes\mathsf{L}_n^*(\mu)\right)\right)^{\Uq(\mathfrak{q}_n)}\cong \delta_{\lambda,\mu}2^{\delta(\lambda)}\mathsf{L}_r^*(\lambda)\otimes \mathsf{L}_s(\lambda).$$

Recall from \cite{GJKKK15} that $\mathsf{L}_r^*(\lambda)\otimes \mathsf{L}_n(\lambda)$ (resp. $\mathsf{L}_s(\mu)\otimes \mathsf{L}_n^*(\mu)$) is the direct sum of $2^{\delta(\lambda)}$ (resp. $2^{\delta(\mu)}$) copies of $\mathsf{L}_r^*(\lambda)\circledast \mathsf{L}_n(\lambda)$ (resp. $\mathsf{L}_s(\mu)\circledast \mathsf{L}_n^*(\mu)$). 
Hence, 
$$2^{\delta(\lambda)+\delta(\mu)}\left(\left(\mathsf{L}_r^*(\lambda)\circledast\mathsf{L}_n(\lambda)\right)\otimes\left(\mathsf{L}_s(\mu)\circledast\mathsf{L}_n^*(\mu)\right)\right)^{\Uq(\mathfrak{q}_n)}\cong \delta_{\lambda,\mu}2^{\delta(\lambda)+\delta(\lambda)}\mathsf{L}_r^*(\lambda)\circledast \mathsf{L}_s(\lambda),$$
which yields
$$\left(\left(\mathsf{L}_r^*(\lambda)\circledast\mathsf{L}_n(\lambda)\right)\otimes\left(\mathsf{L}_s(\mu)\circledast\mathsf{L}_n^*(\mu)\right)\right)^{\Uq(\mathfrak{q}_n)}\cong \delta_{\lambda,\mu}\mathsf{L}_r^*(\lambda)\circledast \mathsf{L}_s(\lambda)$$
Hence, 
\begin{align*}
\left(\mathcal{O}_{r,s}\right)^{\Uq(\mathfrak{q}_n)}
=&\bigoplus_{\lambda\in\mathrm{SP}\atop \ell(\lambda)\leqslant\min(r,n), \ell(\mu)\leqslant\min(s,n)}\left(\left(\mathsf{L}_r^*(\lambda)\circledast\mathsf{L}_n(\lambda)\right)\otimes\left(\mathsf{L}_s(\mu)\circledast\mathsf{L}_n^*(\mu)\right)\right)^{\Uq(\mathfrak{q}_n)}\\
=&\bigoplus_{\lambda\in\mathrm{SP}\atop \ell(\lambda)\leqslant\min(r,s,n)}\left(\mathsf{L}_r^*(\lambda)\circledast \mathsf{L}_s(\lambda)\right).
\end{align*}
This proves the decomposition~\eqref{eq:InvDec}.
\end{proof}

\section{The First Fundamental Theorem}
\label{FFT}

This section is devoted to demonstrating the first fundamental theorem in the invariant theory of $\Uq(\mathfrak{q}_n)$ that describes the generators of the $\Uq(\mathfrak{q}_n)$-invariant subalgebra $\left(\mathcal{O}_{r,s}\right)^{\Uq(\mathfrak{q}_n)}$.

We first consider the case where $n\geqslant\max(r,s)$. In this situation, $\Uq(\mathfrak{q}_r)$ (resp. $\Uq(\mathfrak{q}_s)$) can be regarded as a sub-superalgebra of $\Uq(\mathfrak{q}_n)$ generated by $L_{ij}$ with $i,j\in I_{r|r}$ (resp. $i,j\in I_{s|s}$). As shown in \cite{CW20}, the sub-superalgebra $\mathsf{A}_{r,s}$ of $\mathsf{A}_n$ generated by $1, t_{i\alpha}$ with $i\in I_{r|r}$ and $\alpha\in I_{s|s}$ is a $\Uq(\mathfrak{q}_r)\otimes\Uq(\mathfrak{q}_s)$ and admits a multiplicity free decomposition same as the decomposition in \eqref{eq:InvDec}. This suggest us to connect $\mathsf{A}_{r,s}$ with $\left(\mathcal{O}_{r,s}\right)^{\Uq(\mathfrak{q}_n)}$.

Following the strategy of \cite{LZZ11}, the superalgebra $\mathsf{A}_n$ is also a bi-superalgebra with comultiplication
$$\Delta^{\circ}(t_{i\alpha})=\sum_{b\in I_{n|n}}(-1)^{(|i|+|b|)(|b|+|\alpha|)}t_{ib}\otimes t_{b\alpha},\qquad i,\alpha\in I_{n|n},$$
and $\mathcal{S}^{\circ}(\mathsf{A}_n)=\bar{\mathsf{A}}_n$. Hence, we have a $\mathbb{C}(q)$-linear map
$$\tilde{\Delta}=(1\otimes\mathcal{S}^{\circ})\circ\Delta^{\circ}: \quad\mathsf{A}_n\rightarrow \mathsf{A}_n\otimes\bar{\mathsf{A}}_n.$$
Provided that $n\geqslant\max(r,s)$, the $\mathbb{C}(q)$-linear map $\tilde{\Delta}$ maps the sub-superalgebra $\mathsf{A}_{r,s}$ into the subspace $\mathsf{A}_{r,n}\otimes\bar{\mathsf{A}}_{s,n}=\mathcal{O}_{r,s}$. 
We denote the restriction of $\tilde{\Delta}$ on $\mathsf{A}_{r,s}$ by 
\begin{equation}
\tilde{\Delta}_{r,s}:\quad\mathsf{A}_{r,s}\rightarrow \mathcal{O}_{r,s}.
\end{equation}
The same arguments as \cite[Lemma~6.11]{LZZ11} shows that
\begin{lemma}
\label{lem:injective}
If $n\geqslant\max(r,s)$, then the $\mathbb{C}(q)$-linear map $\tilde{\Delta}_{r,s}$ is injective,
\begin{equation}
\tilde{\Delta}_{r,s}(t_{i\alpha})=x_{i\alpha},
\end{equation}
for $i=1,\ldots, r$ and $\alpha\in I_{s|s}$, and $\tilde{\Delta}_{r,s}(\mathsf{A}_{r,s})\subseteq\left(\mathcal{O}_{r,s}\right)^{\Uq(\mathfrak{q}_n)}$.
\qed
\end{lemma}

\begin{lemma}
Suppose that $n\geqslant\max(r,s)$. If $\mathsf{A}_{r,s}$ and $\mathcal{O}_{r,s}$ are regarded as $\Uq(\mathfrak{q}_r)$-supermodules under the action $\Psi$ and $\Psi\otimes1$, respectively, then the $\mathbb{C}(q)$-linear map $\tilde{\Delta}_{r,s}$ is a homomorphism of $\Uq(\mathfrak{q}_r)$-supermodules.
\end{lemma}
\begin{proof}
For $f\in\mathsf{A}_{r,s}$ and $u\in\Uq(\mathfrak{q}_r)$,
$$\Psi_u(f)=\sum\limits_{(f)}(-1)^{|f_{(1)}||u|}\langle f_{(1)}, \mathcal{S}(u)\rangle f_{(2)}.$$
Then we deduce that
\begin{align*}
\tilde{\Delta}_{r,s}\circ\Psi_u(f)
=&\sum\limits_{(f)}(-1)^{|f_{(1)}||u|}\langle f_{(1)},\mathcal{S}(u)\rangle \tilde{\Delta}_{r,s}(f_{(2)})\\
=&\sum\limits_{(f)}(-1)^{|f_{(1)}||u|}\langle f_{(1)},\mathcal{S}(u)\rangle f_{(2)}\otimes\mathcal{S}^{\circ}(f_{(3)}),\\
(\Psi_u\otimes1)\tilde{\Delta}_{r,s}(f)
=&\sum_{(f)}\Psi_u(f_{(1)})\otimes \mathcal{S}^{\circ}(f_{(2)})\\
=&\sum\limits_{(f)}(-1)^{|f_{(1)}||u|}\langle f_{(1)},\mathcal{S}(u)\rangle f_{(2)}\otimes\mathcal{S}^{\circ}(f_{(3)}).
\end{align*}
It shows $\tilde{\Delta}_{r,s}\circ\Psi_u(f)=(\Psi_u\otimes1)\circ\tilde{\Delta}_{r,s}(f)$, and hence $\tilde{\Delta}_{r,s}$ is a homomorphism of $\Uq(\mathfrak{q}_r)$-supermodules.
\end{proof}

Note that the $\mathbb{C}(q)$-linear map $\tilde{\Delta}_{r,s}$ is not  a homomorphism of superalgebras. However, we will show that the superalgebra and $\Uq(\mathfrak{q}_s)$-supermodule structures on $\mathsf{A}_{r,s}$ and $\mathcal{O}_{r,s}$ are closely related to each other via the $\mathbb{C}(q)$-linear map $\tilde{\Delta}_{r,s}$.

For the $\Uq(\mathfrak{q}_s)$-supermodule structure, we consider the automorphism $\sigma=\mathcal{S}^2$ of the Hopf superalgebra $\Uq(\mathfrak{q}_s)$, where $\mathcal{S}$ is the antipode of $\Uq(\mathfrak{q}_s)$. The action of $\sigma$ on the generators $L_{ij}$ are given by
$$\sigma(L)=(1\otimes D^{-1})L(1\otimes D),$$
where is the diagonal matrix $D=\sum\limits_{\alpha\in I_{s|s}}q^{2(-1)^{|\alpha|}\alpha}E_{\alpha\alpha}$.

Given an arbitrary $\Uq(\mathfrak{q}_s)$-supermodule $M$, there is another $\Uq(\mathfrak{q}_s)$-super- module $M^{\sigma}$ obtained by twisting the $\Uq(\mathfrak{q}_s)$-action with $\sigma$. Namely, $M^{\sigma}$ has the same underlying vector space with $M$, and the $\Uq(\mathfrak{q}_s)$-action
$$u\underset{\sigma}{\cdot} x=\sigma(u)\cdot x,\qquad u\in\Uq(\mathfrak{q}_s), \ x\in M.$$
In particular, $\mathsf{A}_{r,s}$ is equipped with a new structure of $\Uq(\mathfrak{q}_s)$-supermodule $A_{r,s}^{\sigma}$. We denote the new $\Uq(\mathfrak{q}_s)$-action by $\tilde{\Phi}$, i.e., $\tilde{\Phi}_u=\Phi_{\sigma(u)}$ for $u\in\Uq(\mathfrak{q}_s)$. Then we have the following proposition.

\begin{proposition}
Suppose that $n\geqslant\max(r,s)$. If $\mathrm{A}_{r,s}$ and $\mathcal{O}_{r,s}$ are regarded as $\Uq(\mathfrak{q}_s)$-supermodules under the action $\tilde{\Phi}$ and $\bar{\Psi}$ respectively,  then the $\mathbb{C}(q)$-linear map $\tilde{\Delta}_{r,s}$ is a homomorphism of\, $\Uq(\mathfrak{q}_s)$-supermodules. 
\end{proposition}
\begin{proof}
We need to check $\tilde{\Delta}_{r,s}\circ\tilde{\Phi}_u=(1\otimes \bar{\Psi}_u)\circ\tilde{\Delta}_{r,s}$. 

Using the pairing $\langle\mathsf{A}_n,\Uq(\mathfrak{q}_n)\rangle$, the action $\tilde{\Phi}$ of $\Uq(\mathfrak{q}_s)$ on $\mathsf{A}_{r,s}$ and the action $\bar{\Psi}$ of $\Uq(\mathfrak{q}_s)$ on $\bar{\mathsf{A}}_{s,n}$ can be written as
$$\tilde{\Phi}_u(f)=\sum\limits_{(f)}(-1)^{|u||f|}f_{(1)}\langle f_{(2)}, \sigma(u)\rangle,\quad \bar{\Psi}_u(\bar{g})=\sum\limits_{(\bar{g})}(-1)^{|\bar{g}_{(1)}||u|}\langle \bar{g}_{(1)}, \mathcal{S}(u)\rangle\bar{g}_{(2)},$$
for $u\in\Uq(\mathfrak{q}_s)$, $f\in\mathsf{A}_{r,n}$ and $\bar{g}\in\bar{\mathsf{A}}_{s,n}$. Since $$\Delta^{\circ}\left(\mathcal{S}^{\circ}(f)\right)
=\sum_{(f)}(-1)^{|f_{(1)}||f_{(2)}|}\mathcal{S}^{\circ}(f_{(2)})\otimes\mathcal{S}^{\circ}(f_{(1)}),$$ we have
\begin{align*}
\bar{\Psi}_u(\mathcal{S}^{\circ}(f))
=&\sum_{(f)}(-1)^{|f_{(1)}||f_{(2)}|+|f_{(2)}||u|}
\langle\mathcal{S}^{\circ}(f_{(2)}),\mathcal{S}(u)\rangle \mathcal{S}^{\circ}(f_{(1)})\\
=&\sum_{(f)}(-1)^{|f_{(1)}||f_{(2)}|+|f_{(2)}||u|} \mathcal{S}^{\circ}(f_{(1)})\langle f_{(2)},\mathcal{S}^2(u)\rangle\\
=&\sum_{(f)}(-1)^{|f||u|} \mathcal{S}^{\circ}(f_{(1)})\langle f_{(2)},\sigma(u)\rangle.
\end{align*}
Then we compute that
\begin{align*}
\tilde{\Delta}_{r,s}\circ\tilde{\Phi}_u(f)
=&\tilde{\Delta}_{r,s}\left(\sum\limits_{(f)}(-1)^{|u||f|}f_{(1)}\langle f_{(2)}, \sigma(u)\rangle\right)\\
=&\sum\limits_{(f)}(-1)^{|u||f|}f_{(1)}\otimes \mathcal{S}^{\circ}(f_{(2)}) \langle f_{(3)} , \sigma(u)\rangle,\\
=&\sum\limits_{(f)}(-1)^{|u||f_{(1)}|}f_{(1)}\otimes \Psi_u\left(\mathcal{S}^{\circ}(f_{(2)})\right)
=(1\otimes\bar{\Psi}_u)\tilde{\Delta}_{r,s}(f),
\end{align*}
which completes the proof.
\end{proof}

\begin{remark}
The same results also holds in the cases of the general linear group  \cite[Section~6 ]{LZZ11} and the general linear supergroup \cite[Section~4]{Zhangyang20}.
\end{remark}

Next,  we consider the behavior of $\tilde{\Delta}_{r,s}$ when taking product in $\mathsf{A}_{r,s}$. Since $\tilde{\Delta}_{r,s}(fg)$ is not necessarily equal to $\tilde{\Delta}_{r,s}(f)\tilde{\Delta}_{r,s}(g)$, $\tilde{\Delta}_{r,s}$ is not a homomorphism of superalgebras, but the multiplication on $\mathsf{A}_{r,s}$ and $\mathcal{O}_{r,s}$ are connected to each other via $\tilde{\Delta}$. Such a connection could be clarified using the following ``$R$-matrix'' $\Omega$ on $\mathsf{A}_{r,s}\otimes \mathsf{A}_{r,s}$.

\begin{proposition}
\label{prop:Omega}
There exists a $\mathbb{C}(q)$-linear map 
\begin{equation}
\Omega: \mathsf{A}_{r,s}\otimes\mathsf{A}_{r,s}\rightarrow\mathsf{A}_{r,s}\otimes \mathsf{A}_{r,s}
\end{equation}
such that
\begin{equation}
\Omega(f\otimes1)=f\otimes 1, \quad \Omega(1\otimes f)=1\otimes f,\quad
\Omega\left(T_+^{1[3]}T_+^{2[4]}\right)=T_+^{1[3]}T_+^{2[4]}S^{12},
\label{eq:Omega1}
\end{equation}
for $f\in \mathsf{A}_{r,s}$ and
\begin{equation}
\Omega\circ(\mathrm{mul}\otimes1)=(\mathrm{mul}\otimes1)\Omega^{[1][3]}\Omega^{[2][3]},\quad \Omega\circ(1\otimes\mathrm{mul})=(1\otimes \mathrm{mul})\Omega^{[1][3]}\Omega^{[1][2]}.\label{eq:Omega2}
\end{equation}
\end{proposition}

\begin{proof}
Recall from Lemma~\ref{lem:QCAP2} that $\mathsf{A}_{r,s}$ is generated by $1$ and $t_{i,\alpha}$ with $i=1,\ldots, r$ and $\alpha\in I_{s|s}$, the equality \eqref{eq:Omega1} determines $\Omega(f\otimes g)$ when $f,g$ are those generators of $\mathsf{A}_{r,s}$. Then the definition of $\Omega$ is extended to the whole vector $\mathsf{A}_{r,s}\otimes\mathsf{A}_{r,s}$ according to \eqref{eq:Omega2}. It suffices to show $\Omega$ is well-defined.
Thus, we need to check the following two equations
\begin{align}\label{eq:wedOmega1}
\Omega\left(R^{12}T_+^{1[4]}T_+^{2[4]}T_+^{3[5]}\right)=\Omega\left(T_+^{2[4]}T_+^{1[4]}S^{12}T_+^{3[5]}\right),
\end{align}
and 
\begin{align}\label{eq:wedOmega2}
\Omega\left(T_+^{1[4]}R^{23}T_+^{2[5]}T_+^{3[5]}\right)=\Omega\left(T_+^{1[4]}T_+^{3[5]}T_+^{2[5]}S^{23}\right).
\end{align}
By \eqref{eq:Omega2}, we have 
\begin{align*}
\Omega\left(R^{12}T_+^{1[4]}T_+^{2[4]}T_+^{3[5]}\right)
=&R^{12}\Omega\circ(\mathrm{mul}\otimes1)\left( T_+^{1[4]}T_+^{2[5]}T_+^{3[6]}\right)\\
=&R^{12}(\mathrm{mul}\otimes1)\Omega^{[4][6]}\Omega^{[5][6]}\left( T_+^{1[4]}T_+^{2[5]}T_+^{3[6]}\right)\\
=&R^{12}(\mathrm{mul}\otimes1)\Omega^{[4][6]}\left( T_+^{1[4]}T_+^{2[5]}T_+^{3[6]}S^{23}\right)\\
=&R^{12}(\mathrm{mul}\otimes1)\left( T_+^{1[4]}T_+^{2[5]}T_+^{3[6]}S^{13}S^{23}\right)\\
=&R^{12}T_+^{1[4]}T_+^{2[4]}T_+^{3[5]}S^{13}S^{23}\\
=&T_+^{2[4]}T_+^{1[4]}S^{12}T_+^{3[5]}S^{13}S^{23}\\
=&T_+^{2[4]}T_+^{1[4]}T_+^{3[5]}S^{12}S^{13}S^{23}.
\end{align*}
On the other hand,
\begin{align*}
\Omega\left(T_+^{2[4]}T_+^{1[4]}S^{12}T_+^{3[5]}\right)
=&\Omega\circ(\mathrm{mul}\otimes 1)\left(T_+^{2[4]}T_+^{1[5]}T_+^{3[6]}\right)S^{12}\\
=&(\mathrm{mul}\otimes 1)\Omega^{[4][6]}\Omega^{[5][6]}\left( T_+^{2[4]}T_+^{1[5]}T_+^{3[6]}\right)S^{12}\\
=&(\mathrm{mul}\otimes 1)\Omega^{[4][6]}\left(T_+^{2[4]}T_+^{1[5]}T_+^{3[6]}\right)S^{13}S^{12}\\
=&(\mathrm{mul}\otimes 1)\left(T_+^{2[4]}T_+^{1[5]}T_+^{3[6]}\right)S^{23}S^{13}S^{12}\\
=&T_+^{2[4]}T_+^{1[4]}T_+^{3[5]}S^{23}S^{13}S^{12}.
\end{align*}
The equality \eqref{eq:wedOmega1} follows from the matrix $S$ satisfies quantum Yang-Baxter equation. Similarly, one can prove \eqref{eq:wedOmega2}.
\end{proof}

\begin{remark}
The operator $\Omega$ plays the role of an $R$-matrix on the $\Uq(\mathfrak{q}_s)$-super- module $\mathsf{A}_{r,s}\otimes\mathsf{A}_{r,s}$. Such an operator for $\Uq(\mathfrak{gl}_n)$ or $\Uq(\mathfrak{gl}_{m|n})$ is naturally given by the canonical image of the universal $R$-matrix. Proposition~\ref{prop:Omega} explicitly defines $\Omega$ for $\Uq(\mathfrak{q}_n)$ since the universal $R$-matrix for $\Uq(\mathfrak{q}_n)$ does not exist.
\end{remark}

\begin{lemma}
\label{lem:OmegaExp}
For $1\leqslant p<k$, we have
\begin{align*}
&\Omega\left(T_+^{1[a]}\cdots T_+^{p,[a]}T_+^{p+1,[b]}\cdots T_+^{k[b]}\right)\\
=&T_+^{1[a]}\cdots T_+^{p,[a]}T_+^{p+1,[b]}\cdots T_+^{k[b]}
S^{(p,k)}S^{(p,k-1)}\cdots S^{(p,p+1)},
\end{align*}
where $S^{(p,j)}=S^{1j}S^{2j}\cdots S^{pj}$ for $j=p+1,\ldots,k$.
\end{lemma}
\begin{proof}
The case where $p=1$ and $k=2$ is given by \eqref{eq:Omega1}, we prove the statement by induction.

Firstly, we consider the case where $k=p+1$. By \eqref{eq:Omega2}, we have 
\begin{align*}
	&\Omega\left(T_+^{1[a]}T_{+}^{2[a]}\cdots T_{+}^{p-1,[a]}T_+^{p[a]}T_+^{p+1,[b]}\right)\\
	=&\Omega\circ( \mathrm{mul}\otimes 1)\left(T_+^{1[a]}T_{+}^{2[a]}\cdots T_{+}^{p-1,[a]}T_+^{p[b]}T_+^{p+1,[c]}\right)\\
	=&(\mathrm{mul}\otimes1)\Omega^{[a][c]}\Omega^{[b][c]}\left(T_+^{1[a]}T_{+}^{2[a]}\cdots T_{+}^{p-1,[a]}T_+^{p[b]}T_+^{p+1,[c]}\right)\\
	=&(\mathrm{mul}\otimes1)\Omega^{[a][c]}\left(T_+^{1[a]}T_{+}^{2[a]}\cdots T_{+}^{p-1,[a]}T_+^{p[b]}T_+^{p+1,[c]}\right)S^{p,p+1}\\
=&(\mathrm{mul}\otimes1)\left(T_+^{1[a]}T_{+}^{2[a]}\cdots T_{+}^{p-1,[a]}T_+^{p[b]}T_+^{p+1,[c]}\right)S^{(p-1,p+1)}S^{p,p+1}\\
=&T_+^{1[a]}T_{+}^{2[a]}\cdots T_{+}^{p-1,[a]}T_+^{p[a]}T_+^{p+1,[b]}S^{(p,p+1)}.
\end{align*}
For $k\geqslant p+2$, we deduce by the second equality in \eqref{eq:Omega2} that
\begin{align*}
&\Omega\left(T_+^{1[a]}\cdots T_+^{p[a]}T_+^{p+1,[b]}\cdots T_+^{k-1,[b]}T_+^{k[b]}\right)\\
=&\Omega(1\otimes\mathrm{mul})\left(T_+^{1[a]}\cdots T_+^{p[a]}T_+^{p+1,[b]}\cdots T_+^{k-1,[b]}T_+^{k[c]}\right)\\
=&(1\otimes\mathrm{mul})\Omega^{[a][c]}\Omega^{[a][b]}\left(T_+^{1[a]}\cdots T_+^{p[a]}T_+^{p+1,[b]}\cdots T_+^{k-1,[b]}T_+^{k[c]}\right)\\
=&(1\otimes\mathrm{mul})\Omega^{[a][c]}\left(T_+^{1[a]}\cdots T_+^{p[a]}T_+^{p+1,[b]}\cdots T_+^{k-1,[b]}T_+^{k[c]}\right)S^{(p,k-1)}\cdots S^{(p,p+1)}\\
=&(1\otimes\mathrm{mul})\left(T_+^{1[a]}\cdots T_+^{p[a]}T_+^{p+1,[b]}\cdots T_+^{k-1,[b]}T_+^{k[c]}\right)S^{(p,k)}S^{(p,k-1)}\cdots S^{(p,p+1)}\\
=&T_+^{1[a]}\cdots T_+^{p[a]}T_+^{p+1,[b]}\cdots T_+^{k-1,[b]}T_+^{k[b]}S^{(p,k)}S^{(p,k-1)}\cdots S^{(p,p+1)}.
\end{align*}
It completes the proof.
\end{proof}

\begin{proposition}\label{prop:DeltaMul}
The $\mathbb{C}(q)$-linear map $\tilde{\Delta}_{r,s}:\ \mathsf{A}_{r,s}\rightarrow \mathcal{O}_{r,s}$ satisfies
\begin{equation}\label{eq:tildeDelta}
\tilde{\Delta}_{r,s}(fg)=\mathrm{mul}\circ\left(\tilde{\Delta}_{r,s}\otimes\tilde{\Delta}_{r,s}\right)\circ\Omega(f\otimes g).
\end{equation}
\end{proposition}
\begin{proof}
We first observe that 
\begin{align*}
\tilde{\Delta}(f)=\mathrm{mul}\circ\left(\tilde{\Delta}_{r,s}\otimes\tilde{\Delta}_{r,s}\right)\circ\Omega(f\otimes1)
=\mathrm{mul}\circ\left(\tilde{\Delta}_{r,s}\otimes\tilde{\Delta}_{r,s}\right)\circ\Omega(1\otimes f),
\end{align*}
which shows \eqref{eq:tildeDelta} holds when $f=1$ or $g=1$.

The definition of $\tilde{\Delta}_{r,s}$ implies that
\begin{align*}
\tilde{\Delta}_{r,s}\left(T_+^{1[a]}T_+^{2[a]}\cdots  T_+^{k[a]}\right)
=&T_+^{1[a]}T_+^{2[a]}\cdots T_+^{k[a]}\bar{T}^{k[b]}\cdots \bar{T}^{2[b]} \bar{T}^{1[b]}.
\end{align*}
On the other hand, Lemma~\ref{lem:OmegaExp} yields that
\begin{align*}
&\mathrm{mul}\circ\left(\tilde{\Delta}_{r,s}\otimes\tilde{\Delta}_{r,s}\right)\circ\Omega\left(T_+^{1[a]}\cdots T_+^{p[a]}T_+^{p+1,[a']}\cdots T_+^{k[a']}\right)\\
=&\mathrm{mul}\left(\tilde{\Delta}_{r,s}\left(T_+^{1[a]}\cdots T_+^{p[a]}\right)\tilde{\Delta}_{r,s}\left(T_+^{p+1,[b]}\cdots T_+^{k[b]}\right)\right)S^{(p,k)}S^{(p,k-1)}\cdots S^{(p,p+1)}\\
=&\mathrm{mul}\left(T_+^{1[a]}\cdots T_+^{p[a]}\bar{T}^{p[a']}\cdots\bar{T}^{1[a']}T_+^{p+1,[b]}\cdots T_+^{k[b]}\bar{T}^{k[b']}\cdots\bar{T}^{p+1[b']}\right)\\
&\cdot S^{(p,k)}S^{(p,k-1)}\cdots S^{(p,p+1)}.
\end{align*}
Since the multiplication in $\mathcal{O}_{r,s}$ satisfies
$$\mathrm{mul}(\bar{T}^{i[a']}T_+^{j[b]})=T_+^{j[a]}\left(S^{-1}\right)^{ij}\bar{T}^{i[a']}, \quad i<j,$$
and
$$\left(S^{-1}\right)^{ij}\bar{T}^{i[a']}\bar{T}^{j[b']}=\bar{T}^{j[b']}\bar{T}^{j[b']}\left(S^{-1}\right)^{ij},\quad i<j,$$
we deduce that
$$\mathrm{mul}\circ\left(\tilde{\Delta}_{r,s}\otimes\tilde{\Delta}_{r,s}\right)\circ\Omega\left(T_+^{1[a]}\cdots T_+^{p[a]}T_+^{p+1,[b]}\cdots T_+^{k[b]}\right)
=\tilde{\Delta}_{r,s}\left(T_+^{1[a]}T_+^{2[a]}\cdots  T_+^{k[a]}\right).$$
This implies that \eqref{eq:tildeDelta} holds when $f$ and $g$ are monomials in $t_{i\alpha}$'s. Now, $\mathsf{A}_{r,s}$ is spanned by all monomials in $t_{i\alpha}$'s, thus \eqref{eq:tildeDelta} holds for all $f,g\in\mathsf{A}_{r,s}$.
\end{proof}

\begin{remark}
Using \eqref{eq:tildeDelta}, we can explicitly formulate the images of all monomials $t_{i_1,\alpha_1}\cdots t_{i_k,\alpha_k}\in\mathsf{A}_{r,s}$ under the map $\tilde{\Delta}_{r,s}$. 
\begin{equation}
	\tilde{\Delta}_{r,s}\left(T_+^{1[a]}T_{+}^{2[a]}\cdots T_{+}^{k[a]}\right)=X^{1[a]}X^{2[a]}\cdots X^{k[a]}S^{(1,2)}S^{(1,3)}\cdots S^{(1,k)}.\label{eq:genDelta}
\end{equation}
\end{remark}

Now, we state our first main result, which is the quantum analogue of the first fundamental theorem of invariant theory for the quantum queer superalgebra.
\begin{theorem}[FFT for $\Uq(\mathfrak{q}_n)$]\label{thm:FFT}
The invariant sub-superalgebra $\left(\mathcal{O}_{r,s}\right)^{\mathrm{U}_q(\mathfrak{q}_n)}$ is generated by elements $x_{i\alpha}$ with $i=1,\ldots,r$ and $\alpha\in I_{s|s}$. 
\end{theorem}
\begin{proof}
The proof of this theorem splits into the following two cases: $n\geqslant \text{max}(r,s)$ or otherwise. 

Case 1: $n\geqslant\max(r,s)$.

In this situation, we have shown in Lemma~\ref{lem:injective} that $\tilde{\Delta}_{r,s}:\mathsf{A}_{r,s}\rightarrow\mathcal{O}_{r,s}$ is an injective linear map. Moreover, if $\mathsf{A}_{r,s}$ (resp. $\mathcal{O}_{r,s}$) is regarded as the $\Uq(\mathfrak{q}_r)\otimes\Uq(\mathfrak{q}_s)$-supermodule via the action $\Psi\otimes\tilde{\Phi}$ (resp. $\Psi\otimes\bar{\Psi}$), the linear map $\tilde{\Delta}_{r,s}$ is also a homomorphism of $\Uq(\mathfrak{q}_r)\otimes\Uq(\mathfrak{q}_s)$-supermodules. 

As a $\Uq(\mathfrak{q}_r)\otimes\Uq(\mathfrak{q}_s)$-supermodule under the action $\Psi\otimes\Phi$, the superalgebra $\mathsf{A}_{r,s}$ admits the multiplicity free decomposition
\begin{equation}
\mathsf{A}_{r,s}=\bigoplus_{\lambda\in\mathrm{SP}\atop\ell(\lambda)\leqslant\min(r,s)}\left(\mathsf{L}_r^*(\lambda)\circledast\mathsf{L}_s(\lambda)\right).\label{eq:FFTL}
\end{equation}
The action $\tilde{\Phi}$ of $\Uq(\mathfrak{q}_s)$ on $\mathsf{A}_{r,s}$ is obtained by twisting $\Phi$ via the automorphism $\sigma$ of $\Uq(\mathfrak{q}_s)$. If we twist both-sides of the decomposition by the automorphism $\sigma$, we obtain the same decomposition where $\mathsf{A}_{r,s}$ is the $\Uq(\mathfrak{q}_r)\otimes\Uq(\mathfrak{q}_s)$-supermodule under the action $\Psi\otimes\tilde{\Phi}$.

On the other hand, Theorem~\ref{thm:InvDec} implies that the invariant sub-superalgebra $\left(\mathcal{O}_{r,s}\right)^{\Uq(\mathfrak{q}_n)}$ also has the decomposition
\begin{equation}
\left(\mathcal{O}_{r,s}\right)^{\Uq(\mathfrak{q}_n)}=\bigoplus_{\lambda\in\mathrm{SP}\atop\ell(\lambda)\leqslant\min(r,s)}\left(\mathsf{L}_r^*(\lambda)\circledast\mathsf{L}_s(\lambda)\right)\label{eq:FFTR}
\end{equation}
provided that $n\geqslant\max(r,s)$. 

Now, $\tilde{\Delta}_{r,s}$ is an injective homomorphism of $\Uq(\mathfrak{q}_s)$-supermodules, where $\mathsf{A}_{r,s}$ and $\mathcal{O}_{r,s}$ are under the $\Uq(\mathfrak{q}_s)$-action $\tilde{\Phi}$ and $\bar{\Psi}$, respectively. The decomposition \eqref{eq:FFTL} and \eqref{eq:FFTR} shows that every irreducible highest weight $\Uq(\mathfrak{q}_s)$-supermodule with highest weight $\lambda$ has the same multiplicities in both $\mathsf{A}_{r,s}$ and $\left(\mathcal{O}_{r,s}\right)^{\Uq(\mathfrak{q}_n)}$. Hence, $\tilde{\Delta}_{r,s}$ is surjective.

Although $\tilde{\Delta}_{r,s}$ is not a homomorphism of superalgebras, Proposition~\ref{prop:DeltaMul} shows that $\tilde{\Delta}_{r,s}$ sends a product in $\mathsf{A}_{r,s}$ to a linear combinations of products in $\mathcal{O}_{r,s}$. Since $\mathsf{A}_{r,s}$ is generated by $1$ and $t_{i\alpha}$ with $i=1,\ldots, r$ and $\alpha\in I_{s|s}$ and $\tilde{\Delta}_{r,s}(t_{i\alpha})=x_{i\alpha}$, we conclude that $\left(\mathcal{O}_{r,s}\right)^{\Uq(\mathfrak{q}_n)}$ is generated by $x_{i\alpha}$ with $i=1,\ldots, r$ and $\alpha\in I_{s|s}$.
\medskip

Case 2: $n<\max(r,s)$.
Let $m=\min(n, r,s)$. We follow the strategy of \cite[Theorem~6.10]{LZZ11} that reduces $\mathcal{O}_{r,s}$ to a small superalgebra where the result of Case 1 is applied. 

For $m\leqslant m'$, the sub-superalgebra of $\Uq(\mathfrak{q}_{m'})$ generated by $L_{ij}$ with $i\leqslant j$ and $i, j\in I_{m|m}\subseteq I_{m'|m'}$ is isomorphic to $\Uq(\mathfrak{q}_m)$. Following the standard technique in \cite{LZZ11} and \cite{CW20}, for each $\Uq(\mathfrak{q}_{m'})$-supermodule $\mathsf{M}$, the subspace
$$\mathsf{M}^{[m]}:=\left\{v\in\mathsf{M}|L_{ii}.v=v, m<i\leqslant m'\right\}$$
is a $\Uq(\mathfrak{q}_{m})$-supermodule. 

Note that the superalgebra $\mathcal{O}_{r,s}=\mathsf{A}_{r,n}\otimes\bar{\mathsf{A}}_{s,n}$ is a $\Uq(\mathfrak{q}_r)\otimes\Uq(\mathfrak{q}_s)\otimes\Uq(\mathfrak{q}_n)$-supermodule, we consider the appropriate $\Uq(\mathfrak{q}_m)\otimes\Uq(\mathfrak{q}_m)\otimes\Uq(\mathfrak{q}_n)$-sub- supermodule
$$\mathcal{O}_{r,s}(m)=\mathsf{A}_{r,n}^{[m]}\otimes\bar{\mathsf{A}}_{s,n}^{[m]},$$
where
$$\mathsf{A}_{r,n}^{[m]}:=\left\{f\in\mathsf{A}_{r,s}|\Psi_{L_{ii}}(f)=f, m<i\leqslant r\right\}$$
and
$$\bar{\mathsf{A}}_{s,n}^{[m]}:=\left\{f\in\bar{\mathsf{A}}_{r,s}|\Psi_{L_{ii}}(f)=f, m<i\leqslant s\right\}.$$
It is shown in \cite{CW20} that $\mathsf{A}_{r,n}^{[m]}$ (resp. $\bar{\mathsf{A}}_{s,n}^{[m]}$) is isomorphic $\mathsf{A}_{m,m}$ (resp. $\bar{\mathsf{A}}_{m,m}$). Hence, $$\left(\mathcal{O}_{r,s}(m)\right)^{\Uq(\mathfrak{q}_n)}=\left(\mathcal{O}_{m,m}\right)^{\Uq(\mathfrak{q}_n)}$$
that is generated by $x_{i\alpha}$ with $i=1,\ldots,m$ and $\alpha\in I_{m|m}$ as shown in Case 1.

By Theorem~\ref{thm:InvDec}, the invariant sub-superalgebra has the decomposition
\begin{equation}\left(\mathcal{O}_{r,s}\right)^{\Uq(\mathfrak{q}_n)}=\bigoplus_{\ell(\lambda)\leqslant m}\left(\mathsf{L}_r^*(\lambda)\circledast\mathsf{L}_s(\lambda)\right)\label{eq:leqm1}
\end{equation}
as a $\Uq(\mathfrak{q}_r)\otimes\Uq(\mathfrak{q}_s)$-supermodule under the action $\Psi\otimes\bar{\Psi}$. It follows from \cite[Lemma~4.3]{CW20} that the irreducible highest weight $\Uq(\mathfrak{q}_{m'})$-supermodule $\mathsf{L}_{m'}(\lambda)$ gives the irreducible highest weight $\Uq(\mathfrak{q}_m)$-supermodule $\mathsf{L}_{m'}^{[m]}(\lambda)\cong \mathsf{L}_m(\lambda)$ if $\ell(\lambda)\leqslant m$, and  a highest weight vector in $\mathsf{L}_{m'}(\lambda)$ is also a highest weight vector in $\mathsf{L}^{[m]}_{m'}(\lambda)$. Hence,
$$\Uq(\mathfrak{q}_r).\mathsf{L}_m^*(\lambda)=\mathsf{L}_r^*(\lambda)\ \text{ and }\ \Uq(\mathfrak{q}_s).\mathsf{L}_m(\lambda)=\mathsf{L}_s(\lambda).$$
It follows that
\begin{align*}
\left(\mathcal{O}_{r,s}\right)^{\Uq(\mathfrak{q}_n)}
=&\bigoplus_{\lambda\in\mathrm{SP}\atop\ell(\lambda)\leqslant m}\left(\mathsf{L}_r^*(\lambda)\circledast\mathsf{L}_s(\lambda)\right)\\
=&\left(\Uq(\mathfrak{q}_r)\otimes\Uq(\mathfrak{q}_s)\right).\left(\bigoplus_{\lambda\in\mathrm{SP}\atop\ell(\lambda)\leqslant m}\mathsf{L}_m^*(\lambda)\circledast\mathsf{L}_m(\lambda)\right)\\
=&\left(\Uq(\mathfrak{q}_r)\otimes\Uq(\mathfrak{q}_s)\right).\left(\mathcal{O}_{r,s}(m)\right)^{\Uq(\mathfrak{q}_n)},
\end{align*}
which shows $\left(\mathcal{O}_{r,s}\right)^{\Uq(\mathfrak{q}_n)}$ is generated by $\left(\Uq(\mathfrak{q}_r)\otimes\Uq(\mathfrak{q}_s)\right). x_{i\alpha}$ with $i=1,\ldots, m$ and $\alpha\in I_{m|m}$. 
Therefore, $\left(\mathcal{O}_{r,s}\right)^{\Uq(\mathfrak{q}_n)}$ is generated by $x_{i\alpha}$ with $i=1,\ldots, r$ and $\alpha\in I_{s|s}$ since 
$$\Uq(\mathfrak{q}_r)\underset{\Psi}{\cdot}x_{i\alpha}=\mathrm{span}\left\{x_{j\alpha},\ x_{j,-\alpha}| \ j=1,\ldots,r\right\},$$
and
$$\Uq(\mathfrak{q}_s)\underset{\Psi}{\cdot}x_{i\alpha}=\mathrm{span}\left\{x_{i\beta}|\ \beta\in I_{s|s}\right\}.$$
This completes the proof in the case $n<\max(r,s)$.
\end{proof}
\begin{remark}
Sergeev investigated the invariant theory for classical Lie superalgebra. Explicitly, there are two families inavriants for queer Lie superalgebra which corresponding to $x_{ia}$ for $i=1,\ldots,r, a=1,\ldots s$ and  $x_{ia}$ for $i=1,\ldots,r, a=-1,\ldots -s$, see \cite[Theorem~1.5.1]{Sergeev01}.
\end{remark}

\section*{Acknowledgments}
We would like to express our debt to Prof. Shun-Jen Cheng, and Ruibin Zhang for many insightful discussions.This paper was partially written up during the second author visit to School of Mathematics, Wu Wen-Tsun Key Laboratory of Mathematics, USTC, in Summer 2022, from which we gratefully acknowledge the support and excellent working environment where most of this work was completed. Z. Chang is  partially supported by the National Natural Science Foundation of China (No. 12071150), Guangdong Basic and Applied Basic Research Foundation (No. 2020A1515011417), and Science and Technology Planning Project of Guangzhou (No. 202102021204). Y. Wang is partially supported by the National Natural Science Foundation of China (Nos. 11901146 and 12071026), and the Fundamental Research Funds for the Central Universities JZ2021HGTB0124.

\end{document}